\documentclass[12pt]{article}

\usepackage{amsthm,amsmath,stmaryrd,bbm,geometry,color}
\usepackage{amssymb}
\usepackage[english]{babel}
\usepackage[utf8]{inputenc}
\usepackage{graphicx}
\usepackage{verbatim}
\usepackage{enumitem}
\usepackage{mathtools}
\usepackage[titletoc]{appendix}
\usepackage{bm}
\usepackage{authblk}

\usepackage[dvipsnames]{xcolor}
\usepackage[hyperindex=true,frenchlinks=true,colorlinks=true,
citecolor=Mahogany,linkcolor=Red,urlcolor=Tan,linktocpage]{hyperref}

\usepackage{tikz}
\usetikzlibrary{shapes}
\usetikzlibrary{fit}
\usetikzlibrary{decorations.pathmorphing}

\title{Deviation inequality for Banach-valued
orthomartingales}
\author{Davide Giraudo}

\date{\today}
\affil[$\dagger$]{Institut de Recherche Mathématique Avancée
UMR 7501, Université de Strasbourg and CNRS
7 rue René Descartes
67000 Strasbourg, France}

\numberwithin{equation}{section}
\renewcommand{\leq}{\leqslant}
\renewcommand{\geq}{\geqslant}

\newtheorem{Theorem}{Theorem}[section]
\newtheorem{Proposition}[Theorem]{Proposition}
\newtheorem{Lemma}[Theorem]{Lemma}
\newtheorem{Definition}[Theorem]{Definition}
\newtheorem{Corollary}[Theorem]{Corollary}

\theoremstyle{remark}
\newtheorem{Remark}[Theorem]{Remark}

\tikzstyle{Vertex}=[circle,draw=LimeGreen!80,fill=LimeGreen!8,
inner sep=1pt,minimum size=2mm,line width=1pt,font=\scriptsize]
\tikzstyle{Node}=[Vertex,draw=RoyalBlue!80,fill=RoyalBlue!8,inner sep=1.5pt]
\tikzstyle{Leaf}=[rectangle,draw=Black!70,fill=Black!16,
inner sep=0pt,minimum size=1mm,line width=1.25pt]
\tikzstyle{Edge}=[Maroon!80,cap=round,line width=1pt]
\tikzstyle{Mark1}=[draw=BrickRed!80,fill=BrickRed!8]
\tikzstyle{Mark2}=[draw=BurntOrange!80,fill=BurntOrange!8]
\tikzstyle{EdgeRew}=[->,RedOrange!80,cap=round,thick]

\newcommand{\Bca}{\mathcal{B}}

\newcommand{\Fca}{\mathcal{F}}
\newcommand{\Gca}{\mathcal{G}}
\newcommand{\Hca}{\mathcal{H}}

\newcommand{\B}{\mathbb{B}}
\newcommand \ens[1]{\left\{ #1\right\}}

\newcommand \R{\mathbb R}

\newcommand \N{\mathbb N}
\newcommand \PP{\mathbb P}
\newcommand{\el}{\mathbb L}
\newcommand{\E}[1]{\mathbb E\left[#1\right]}

\newcommand \Z{\mathbb Z}

\newcommand \abs[1]{\left|#1\right|}
\newcommand \eps{\varepsilon}

\newcommand{\D}[1]{\mathrm{d}#1}

\newcommand{\pr}[1]{\left(#1\right)}
\newcommand{\norm}[1]{\left\lVert #1 \right\rVert}
\newcommand{\gr}[1]{\bm{#1}}
\newcommand{\imd}{\preccurlyeq}
\newcommand{\smd}{\succcurlyeq}
\newcommand{\grn}{\gr{n}}
\newcommand{\grN}{\gr{N}}
\newcommand{\gru}{\gr{u}}
\newcommand{\gri}{\gr{i}}
\newcommand{\grj}{\gr{j}}
\newcommand{\conv}{\underset{\operatorname{conv}}{\leq}}
\newcommand{\1}[1]{\mathbf{1}\ens{#1}}
\newcommand{\ipr}[1]{\mathbf{1}\pr{#1}}

\newcommand{\ind}[1]{\mathbf{1}_{#1}}
\newcommand{\tail}[1]{\operatorname{Tail}\pr{#1}}

\begin{document}

\maketitle

\begin{abstract}
 We show a deviation inequality inequalities for multi-indexed martingale  We then provide applications to kernel
regression for random fields and rates in the law of large numbers for orthomartingale difference random fields.
\end{abstract}
 
\section{Deviation inequalities for orthomartingale difference random fields}

\subsection{Introduction, motivations and summary of the contribution of the 
paper}

Giving a bound on the tail of a random variable is a fundamental tool in order 
to measure the rates of convergence of a collection of random elements, for 
example in the context of the strong law of large numbers. Such inequalities 
can also be used in order to check tightness criterion. Therefore, a lot of 
attention has been given to the obtention of probability inequalities: in the 
independent case \cite{MR669054,MR144363}, mixing case (chapter 6 in 
\cite{MR2117923}), functions of an i.i.d. sequence \cite{MR3114713} or 
martingales \cite{MR2021875,MR3579898,MR3311214}

In this paper, we will focus on orthomartingale difference random fields, that 
is, a special case of multi-indexed martingale difference random which allows 
to exploit unidimensional martingale difference properties 
and use arguments based on induction on the dimension. 
When the increments are form a strictly stationary random 
field, most of the main limit theorems have been investigated: the central 
limit theorem \cite{MR3427925,MR3913270}, quenched versions 
of the functional central limit theorem \cite{MR4125956} 
and the law of the iterated logarithms \cite{MR4186670}. 
However, the law of large numbers has not been as much investigated as much as 
the other limit theorems. 

We will bring the following contribution to the study of orthomartingale 
difference random fields.
\begin{enumerate}
 \item We establish a deviation inequality for orthomartingale difference 
random field takings values in a separable Banach space that have to satisfy 
some smothness assumptions. Note that we do not assume the random field to 
be identically distributed.
\item Since the random field into consideration can 
have any marginal distribution, provided that it possesses 
the orthomartingale structure, we can consider weighted sums 
of such random fields. This gives the possibility to provide 
applications to regression models.
\item We give also an optimal sufficient condition for 
the law of large numbers of an identically distributed 
orthomartingale difference random field.
\end{enumerate}

\subsection{A deviation inequality for orthomartingale 
difference random fields}

Given a real-valued martingale difference sequence $\pr{D_j,\Fca_{j}}$, it
is possible to control the tail of $\max_{1\leq n\leq N}\abs{\sum_{j=1}^nD_j}$
by a functional of the tails of $\max_{1\leq j\leq N}\abs{D_j}$ and those of
the predictable quadratic variance $\sum_{j=1}^N\E{D_j^2\mid\mathcal F_{j-1}}$,
see \cite{MR2021875}. An analoguous result has been obtained in
\cite{MR4046858} for martingales with values in a smooth Banach space which may
not have a finite moment of order two.

In order to define orthomartingales, we first need to define an order relation
on $\Z^d$. It turns out that the most conveniant one is the coordinatewise
order.
For $\gri=\pr{i_\ell}_{\ell=1}^d,\grj=\pr{j_\ell}_{\ell=1}^d\in\Z^d$, we say
that $\gri\imd\grj$ if for each $\ell\in\ens{1,\dots,d}$, $i_\ell\leq j_\ell$.
Once this order is defined, we can introduce the concept of completely commuting
filtrations.
\begin{Definition}\label{dfn:commutante}
 We say that a collection of $\sigma$-algebras $\pr{\Fca_{\gri}}_{\gri\in\Z^d}$
is a completely commuting filtration if
\begin{enumerate}
 \item for each $\gri,\grj\in\Z^d$ such that $\gri\imd\grj$,
$\Fca_{\gri}\subset\Fca_{\grj}$ and
 \item for each $Y\in\mathbb L^1$ and each $\gri,\grj\in\Z^d$,
 \begin{equation}\label{eq:def_filtration_commutante}
  \E{\E{Y\mid\Fca_{\gri}}\mid\Fca_{\grj}}=
  \E{Y\mid\Fca_{\min\ens{\gri,\grj}}},
 \end{equation}
 where $\min\ens{\gri,\grj}$ is the element of $\Z^d$ defined as the
 coordinatewise minimum of $\gri$ and $\grj$, that is,
 $\min\ens{\gri,\grj}=\pr{\min\ens{i_\ell,j_\ell}}_{\ell=1}^d$.
\end{enumerate}

\end{Definition}

Let us give two examples of commuting filtrations.
\begin{Proposition}\label{prop:exemple_filtra_commutantes}
 \begin{enumerate}
 \item If $\pr{\eps_{\gru}}_{\gru\in\Z^d}$ is i.i.d., then defining
 $\Fca_{\gri}=\sigma\pr{\eps_{\gru},\gru\in\Z^d,\gru\imd\gri}$, the filtration
 $\pr{\Fca_{\gri}}_{\gri\in\Z^d}$ is completely commuting.
 \item Suppose that $\pr{\Fca^{\pr{\ell}}_{\gri^{\pr{\ell}}
}}_{\gri^{\pr{\ell}} \in \Z^{d_\ell}    }$, $1\leq \ell\leq L$, are completely
commuting filtrations on a probability space $\pr{\Omega,\Fca,\PP}$. Suppose
that for each $\gr{i^{\pr{1}}}\in\Z^{d_1},\dots,\gr{i^{\pr{L}}}\in\Z$, the
$\sigma$-algebras $\Fca^{\pr{\ell}}_{\gri^{\pr{\ell}}
}$, $1\leq \ell\leq L$, are independent. Let $d=\sum_{\ell=1}^Ld_\ell$ and
for $\gri=\pr{i^{\pr{\ell}}}_{\ell=1}^L\in\Z^d$, let
$\Fca_{\gri}=\bigvee_{\ell=1}^L\Fca^{\pr{\ell}}_{\gr{i_\ell}}$. Then
$\pr{\Fca_{\gri}}_{\gri\in\Z^d}$ is completely commuting.
\end{enumerate}
\end{Proposition}

Both examples where introduced in Section 1 of \cite{MR420845}, but without
proof. The first item is a direct consequence of Proposition~2 p. 1693 of
\cite{MR3222815}.

We are now in position to define orthomartingale
martingale difference random field, which allows to exploit the martingale
property in every
direction. To formize this, we need to denote by
$e_{\gr{\ell}}$, $\ell\in\ens{1,\dots,d}$,  the element of $\Z^d$ whose
$\ell$-th coordinate is $1$ and all the others are zero.

\begin{Definition}
 Let $\pr{X_{\gri}}_{\gri\in \Z^d}$ be a random field taking values in a
separable Banach space $\pr{\B,\norm{\cdot}_{\B}}$. We say that
$\pr{X_{\gri}}_{\gri\in \Z^d}$ is an orthomartingale
martingale difference random field with respect to the completely commuting
filtration
$\pr{\Fca_{\gri}}_{\gri\in\Z^d}$ if for each $\gri\in\Z^d$,
$\norm{X_i}_{\B}$ is integrable, $X_{\gri}$ is $\Fca_{\gri}$-measurable and for
each $\ell\in\ens{1,\dots,d}$, $\E{X_{\gri}\mid \Fca_{\gri-\gr{e_\ell}}}=0$.
\end{Definition}

Such a definition is very convenient because summation on a rectangular region
of $\Z^d$ can be treated with martingale properties when summing on a fixed
coordinate.

This allows to use induction arguments. For example,
it can be shown using similar arguments as in the proof of Lemma~2.2
in \cite{MR3483738} that combining an induction argument with Theorem~2.1 in \cite{MR2472010}, the
inequality holds
\begin{equation}
\norm{\sum_{\gr{1}\imd \gri\imd
\gr{N}}X_{\gri}}_p
\leq \pr{p-1}^{d/2}\pr{\sum_{\gr{1}\imd \gri\imd
\gr{N}}\norm{X_{\gri}}_{p}^{2}}^{1/2}.
\end{equation}

Moreover, it has been shown in \cite{MR2264866} that for each $p\geq 1$
and each real-valued orthomartingale difference random field that
\begin{equation}\label{eq:Fazekas}
 c_1\pr{p,d}\norm{\sum_{\gr{1}\imd \gri\imd
\gr{N}}X_{\gri}^2}_{p/2}^{1/2}\leq
\norm{\sum_{\gr{1}\imd \gri\imd
\gr{N}}X_{\gri}}_p
\leq c_2\pr{p,d}\norm{\sum_{\gr{1}\imd \gri\imd
\gr{N}}X_{\gri}^2}_{p/2}^{1/2},
\end{equation}
for some constants $ c_1\pr{p,d}$ and $ c_2\pr{p,d}$ depending only on $p$ and
$d$, which extends the two-dimensional result obtained in \cite{MR520006}.

This suggest that the tails of the partial sums of an orthomartingale
difference random field can be controlled by a functional of the tail of
the sum of squares. We plan to formulate such an inequality for Banach-valued
orthomartingale difference random fields. Some assumptions on the geometry of
the Banach space are required.

\begin{Definition}
 Let $\pr{\B,\norm{\cdot}_{\B}}$ be a separable Banach space. We say that $\B$
is $r$-smooth for $1<r\leq 2$ if there exists an equivalent norm
$\norm{\cdot}'_{\B}$ on $\B$ such that
\begin{equation}
 \sup_{t>0}\sup_{x,y\in\B,\norm{x}'_{\B}=\norm{y}'_{\B}=1,}\frac{\norm{x+ty}'_{\B}+\norm{x-ty}'_{\B}-2}{
t^r}<\infty.
\end{equation}

\end{Definition}

For example, if $\mu$ is $\sigma$-finite on the Borel $\sigma$-algebra of $\R$, 
then $\mathbb L^p\pr{\R,\mu}$ is $\min\ens{p,2}$-smooth.
Moreover, a separable Hilbert space is $2$-smooth.
\begin{Definition}[Martingale in Banach spaces]
 Let $\pr{\B,\norm{\cdot}_{\B}}$ be a separable Banach space. We say that a
 sequence $\pr{D_i}_{i\geq 1}$ is a martingale differences sequence with
respect to a filtration $\pr{\Fca_{i}}_{i\geq 0}$ if for each $i\geq 1$,
$D_{i}$ is $\Fca_i$-measurable and for each $A\in\Fca_{i-1}$,
$\E{X_i\ipr{A}}=0$.
\end{Definition}

By \cite{MR0407963}, we know that if $\B$ is a separable $r$-smooth Banach
space, then
there exists a constant $C$ such that for each martingale
difference sequence $\pr{D_i}_{i\geq 1}$
with values in $\B$ and each $n$,
\begin{equation}\label{eq:moments_ordre_r_martingale_Banach}
 \E{\norm{\sum_{i=1}^nD_i}_{\B}^r}\leq C\sum_{i=1}^n\E{\norm{D_i}_{\B}^r}.
\end{equation}
By definition, an $r$-smooth Banach space is also $p$-smooth for $1<p\leq r$,
hence it is possible to define
\begin{equation} \label{eq:definition_constant_Banach_space}
 C_{p,\B}:=\sup_{n\geq 1}\sup_{\pr{D_i}_{i=1}^n \in \Delta_n}
\frac{\E{\norm{\sum_{i=1}^nD_i}_{\B}^p} }{\sum_{i=1}^n\E{\norm{D_i}_{\B}^p}},
\end{equation}
where the $\Delta_n $ denotes the set of the  martingale differences sequences
$\pr{D_i}_{i=1}^n$ such that $\sum_{i=1}^n \norm{D_i}_{\B}^p$ is not
identically $0$.

A key tool to prove deviation and moment inequalities for martingale difference 
sequence is a so-called "good-$\lambda$-inequality", that is, an 
inequality of the form
\begin{equation}
 \PP\pr{X>\beta \lambda, Y\leq \delta\lambda}\leq 
 \delta^p\pr{\beta-\delta-1}^{-p}\PP\pr{X>\lambda}, 
\end{equation}
where $X=\max_{1\leq n\leq N}\norm{\sum_{i=1}^nD_i}_{\B}$ 
and $Y=\max\ens{\max_{1\leq i\leq N}\norm{D_i}_{\B}, 
\pr{\sum_{i=1}^N\E{\norm{D_i}_{\B}^p\mid\Fca_{i-1}} }^{1/p}}$. Such an approach 
was used in order to derive Burkholder's inequality (see 
\cite{MR365692,MR995572} in the real valued case, \cite{MR3077911,MR4046858} in 
the Banach-valued case). 
Usually, a way to obtain a good-$\lambda$-inequality 
is to introduce a martingale transform of the original martingale based on 
stopping times, where  the increments are controlled as 
well as the indices where the maximum lies between $\lambda$ 
and $\beta\lambda$. Unfortunately, such a method 
does not seem to be appliable in the context of multi-indexed martingales 
essentially because there is no proper generalization of stopping times and 
martingale transforms. 

To overcome this problem, we propose an approach by induction on the dimension. 
In the one dimensional case, we control the tail of the maximum of partial sums 
by the sum of $p$-th power of the increments (see 
Proposition~\ref{prop:deviation_martingales_power}). 
This will be also used for the induction step; the resulting 
sum of powers of norms will be considered as the norm of an orthomartingale 
difference random field indexed by $\Z^{d-1}$ in a modified 
version of the original Banach space, but sharing smoothness properties.

We thus get a control of the tail of the  maximum of the partial sums 
over rectangles by a function of the tail of the sum of $p$-th powers of norms 
of the increments, which is our first result.

\begin{Theorem}\label{thm:deviation_orthomartingales}
Let $1<r\leq 2$ and let $\pr{\B,\norm{\cdot}_{\B}}$ be a separable $r$-smooth
Banach space. For each $p\in (1,r]$, $q>p$ and $d\geq 1$, there exists
a function $f_{p,q,d}\colon \R_+\to \R_+$ such that if $\pr{X_{\gri}}_{\gri\in
\Z^d}$ is a an orthomartingale martingale
differences
random field with respect to a completely commuting filtration
$\pr{\Fca_{\gri}}_{\gri\in\Z^d}$, and taking values in a $\B$, then for each
$1<p\leq r$,
$q>p$ and  $x>0$, the following inequality holds:
\begin{multline}\label{eq:deviation_orthomartingale}
 \PP\pr{\max_{\gr{1}\imd \grn\imd \gr{N}}\norm{\sum_{\gr{1}\imd \gri\imd
 \grn }X_{\gri }  }_{\B} >t
 }
  \leq f_{p,q,d}\pr{C_{p,\B}}
 \int_0^1  u^{q-1} \mathbb
P\pr{
 \pr{\sum_{\gr{1}\imd \gri\imd \gr{N}}\norm{X_{\gri}}_{\B}^p}^{\frac 1p}>tu  }
\mathrm{d}u+\\ +
f_{p,q,d}\pr{C_{p,\B}}
 \int_1^\infty u^{p-1}\pr{1+\log u}^{d-1}\mathbb
P\pr{
 \pr{\sum_{\gr{1}\imd \gri\imd \gr{N}}\norm{X_{\gri}}_{\B}^p}^{\frac 1p}>tu  }
\mathrm{d}u.
\end{multline}
\end{Theorem}
Let us make some comments about this result. 

First, observe that the right hand side of \eqref{eq:deviation_orthomartingale} 
is finite if and only if $\E{\norm{X_{\gri}}_{\B}^p
\pr{1+\log\pr{\norm{X_{\gri}}_{\B}}  }^{d-1} }$ is finite, which is barely more 
restrictive than finiteness of $\E{\norm{X_{\gri}}_{\B}^p}$. Second, letting $Y= 
\pr{\sum_{\gr{1}\imd \gri\imd \gr{N}}\norm{X_{\gri}}_{\B}^p}^{\frac 1p}$, one 
can replace the right hand side of \eqref{eq:deviation_orthomartingale} by 
\begin{equation}
 f_{p,q,d}\pr{C_{p,\B}}\E{\pr{\frac Yt}^q\ind{Y\leq t}}+
 f_{p,q,d}\pr{C_{p,\B}}\E{\pr{\frac Yt}^p  
 \pr{1+\log \pr{\frac Yt}}^{d-1}
 \ind{Y>t}}.
\end{equation}
For $s>p$, multiplying by $st^{s-1}$ in 
\eqref{eq:deviation_orthomartingale} (with $q=s+1$) and integrating over the 
positive real line gives the following moment inequality:
\begin{equation}
 \norm{\max_{\gr{1}\imd \grn\imd \gr{N}}\norm{\sum_{\gr{1}\imd \gri\imd
 \grn }X_{\gri }  }_{\B}}_{s}\leq 
 K\pr{p,s,d}\norm{\pr{\sum_{\gr{1}\imd \gri\imd 
\gr{N}}\norm{X_{\gri}}_{\B}^p}^{\frac 1p}}_{s}.
\end{equation}
This gives a partical generalization to the result of 
\cite{MR2264866}, since we provide an analogue of 
the second inequality in \eqref{eq:Fazekas} for Banach-valued orthomartingale 
difference random fields.

One can generalize Theorem~1.13 of \cite{MR4046858} to stochastically dominated
orthomartingale difference random fields. To state it, we need to define
the following order on random variables: we say that $X\conv Y$ for two
real-valued random variables if for each
convex increasing function $\phi\colon\R\to\R$, $\E{\phi\pr{X}}\leq
\E{\phi\pr{Y}}$.

\begin{Corollary}\label{cor:dev_ortho_stoch_dom}
Let $1<r\leq 2$ and let $\pr{\B,\norm{\cdot}_{\B}}$ be a separable $r$-smooth
Banach space. For each $p\in (1,r]$, $q>p$ and $d\geq 1$, there exists
a function $f_{p,q,d}\colon \R_+\to \R_+$ such that if $\pr{X_{\gri}}_{\gri\in
\Z^d}$ is a an orthomartingale martingale
differences
random field with respect to a completely commuting filtration
$\pr{\Fca_{\gri}}_{\gri\in\Z^d}$, taking values in $\B$, and such that
there exists a real-valued random variable
$V$ such that
\begin{equation}
  \sum_{\gr{1}\imd \gri\imd \gr{N}}\norm{X_{\gri}}_{\B}^p \conv V^p,
\end{equation}
 then for each
$1<p\leq r$,
$q>0$ and  $x>0$, the following inequality holds:
\begin{multline}\label{eq:dev_ortho_stoch_dom}
 \PP\pr{\max_{\gr{1}\imd \grn\imd \gr{N}}\norm{\sum_{\gr{1}\imd \gri\imd
 \grn }X_{\gri }  }_{\B} >t
 }  \leq f_{p,q,d}\pr{C_{p,\B}}
 \int_0^1 u^{q-1}  \mathbb P\pr{V>tu  }
\mathrm{d}u\\
+f_{p,q,d}\pr{C_{p,\B}}
 \int_1^\infty  u^{p-1}\pr{1+ \log u}^{d}\mathbb P\pr{
V>tu  }
\mathrm{d}u.
\end{multline}
\end{Corollary}

The functions $f_{p,q,d}$ involved in \eqref{eq:dev_ortho_stoch_dom} are bigger 
than the ones in \eqref{eq:deviation_orthomartingale}. Moreover, the second 
term in the right hand side contains a power $d$ instead of $d-1$, which is due 
to the combination with an other tail inequality under convex ordering (see 
\eqref{eq:conv_ordering_tails_integrales}). In some applications, this will 
play a role.

\section{Applications}

 \subsection{Application to regression models}

 We consider the following regression model:
 \begin{equation}
  Y_{\gri}=g\pr{\frac{\gr{i}}n}+X_{\gr{i}}, \quad
  \gr{i}\in\Lambda_n:=\ens{1,\dots,n}^d,
 \end{equation}
 where $g\colon [0,1]^d\to \R$ is an unknown smooth function
 and $\pr{X_{\gr{i}}}_{\gr{i}\in \Z^d}$ is an orthomartingale difference
random field. Let $K$ be a probability
 kernel defined on $\R^d$ and let $\pr{h_n}_{n\geq 1}$ be a
 sequence of positive numbers which converges to zero and which
 satisfies
 \begin{equation}\label{eq:assumption_nhn}
  \lim_{n\to+\infty}nh_n=+\infty\mbox{ and }\lim_{n\to+\infty}nh_n^{d+1}=0.
 \end{equation}

 We estimate the function
 $g$ by the kernel estimator $g_n$ defined by
 \begin{equation}\label{eq:definition_de_gnx}
 g_n\pr{\gr{x}}=\frac{\sum_{\gr{i}\in\Lambda_n}Y_{\gr{i}}
 K\pr{\frac{\gr{x}-\gr{i}/n}{h_n}}}
 {\sum_{\gr{i}\in\Lambda_n} K\pr{\frac{\gr{x}-\gr{i}/n}{h_n}}}
 ,\quad \gr{x}\in [0,1]^d,
 \end{equation}
where
\begin{equation}\label{eq:definition_de_Lambda_n_regression}
 \Lambda_n=\ens{1,\dots,n}^d.
\end{equation}

 We make the following assumptions on the regression function $g$
 and the probability kernel $K$:

 \begin{enumerate}[label=(A\arabic*)]
  \item\label{itm:assumption1} The probability kernel $K$
  fulfills $ \int_{\R^d}K\pr{\gr{u}}\mathrm d\gr{u}=1$,
  is symmetric, non-negative, supported by $[-1,1]^d$.
  \item\label{itm:assumption_bounds} There
  exist positive constants   $c$ and $C$ such that for any
  $\gr{x} \in [-1,1]^d$,
  $c\leq K\pr{\gr{x}}\leq C$.
  \item\label{itm:assumption2} There exists a positive constant $C$
  such that the absolute values of
  all the derivatives of first order of $g$ are bounded by $K$ on $[0,1]^d$ 	ns 
  for each $\gr{x},\gr{y}\in [0,1]^d$,   $\abs{K\pr{\gr{x}}-K\pr{\gr{y}}} \leq C\norm{\gr{x}-\gr{y}}_\infty$.
 \end{enumerate}

Assumption~\ref{itm:assumption2} will not be used in the
following result. However, by Proposition~1 in \cite{MR2269603},
this guarantees that
\begin{equation}
 \sup_{\gr{x}\in [0,1]^d}\sup_{g\in \operatorname{Lip}\pr{K}}
 \abs{\E{g_n\pr{\gr{x}}}-g\pr{\gr{x}}}=O\pr{h_n},
\end{equation}
where $\operatorname{Lip}\pr{K}$ denotes the collection of
all $K$-Lipschitz functions on $\R^d$.  

\begin{Theorem} \label{thm:LLN_regression}
Let $p>1$ and let $\pr{X_{\gri}}_{\gri\in \Z^d}$ be an identically distributed real-valued orthomartingale difference random field and let
$g_n\colon [0,1]^d\to \R$ be given by \eqref{eq:definition_de_gnx}. Assume that 
\ref{itm:assumption1} and \ref{itm:assumption2} hold.
For each positive $t$,
the following inequality takes place:
\begin{itemize}
 \item for $1<p\leq 2$,
 \begin{multline}\label{eq:regression_p<2}
 \PP\pr{\norm{g_n\pr{\cdot}-\E{g_n\pr{\cdot}} }_{\el^p\pr{[0,1]^d}}>t  }\leq \kappa_{p,q,d}
 \int_0^1u^{q-1}\PP\pr{\abs{X_{\gr{1}}} >t\pr{nh_n}^{d\pr{1-1/p}} u }du\\
 +\kappa_{p,q,d}\int_1^\infty u^{p-1}\pr{1+\log u}^d
 \PP\pr{\abs{X_{\gr{1}}} >t\pr{nh_n}^{d\pr{1-1/p}}u }du;
\end{multline}
 \item for $p>2$,
 \begin{multline}\label{eq:regression_p>2}
 \PP\pr{\norm{g_n\pr{\cdot}-\E{g_n\pr{\cdot}} }_{\el^p\pr{[0,1]^d}}>t  }\leq 
\kappa_{p,q,d}
 \int_0^1u^{q-1}\PP\pr{\abs{X_{\gr{1}}} >t\pr{nh_n}^{d\pr{p-1}/2} 
n^{d\frac{2-p}{2p}} u }du\\
 +\kappa_{p,q,d}\int_1^\infty u\pr{1+\log u}^d
 \PP\pr{\abs{X_{\gr{1}}} >t\pr{nh_n}^{d\pr{p-1}/2} 
n^{d\frac{2-p}{2p}}  u }du.
\end{multline}
\end{itemize}
\end{Theorem}

Note that in the case $1<p\leq 2$, the assumptions that $nh_n\to\infty$ 
and $\E{\abs{X_1}^p\pr{1+ \log\pr{\abs{X_1}}\ind{\abs{X_1}>1}}^{d-1}}$ is finite
 suffice
to guarantee that the right hand side of \eqref{eq:regression_p<2} goes to $0$ 
as $n$ goes to infinity. 
However, in the case $p>2$, the extra term $n^{d\frac{2-p}{2p}}$ imposes 
a restriction in the choice of the bandwidth. For example, if $h_n=n^{-\gamma}$,
then we should have $1/\pr{d+1}<\gamma<1-\pr{p-2}/\pr{p\pr{p-1}}$.

\subsection{Law of large numbers for sums of Banach-valued orthomartingale
difference
random fields over rectangles}\label{subsec:LGN}

In this Subsection, we will deal with convergence rates of
orthomartingale difference random fields.

Althought our first result is not a consequence of
Theorem~\ref{thm:deviation_orthomartingales}, it gives a necessary and
sufficient condition for the Marcinkievicz strong law of large numbers to take
place in a smooth Banach space.

In order to state it, we need, to introduce the following norm.
For $p\geq 1$ and $q\geq 0$, denote by $\norm{\cdot}_{p,q}$ the
Orlicz-norm associated to the Young funtion $\varphi_{p,q}\colon t\in
(0,\infty)]\mapsto t^p\pr{1+\abs{\log t} }^{q}$, that is,
\begin{equation}
 \norm{X}_{p,q}=\inf\ens{\lambda>0\mid \E{\varphi_{p,q}\pr{\frac{\abs{X}
}{\lambda} }} \leq 1 }.
\end{equation}
For $q=0$, $\norm{\cdot}_{p,q}=\norm{\cdot}_{p,0}$ reduces to
the classical $\el^p$-norm and will be simply denoted as
$\norm{\cdot}_p$.

For $\grn=\pr{n_\ell}_{\ell=1}^d\in\N^d$, we define 
$\gr{2^n}=\pr{2^{n_\ell}}_{\ell\geq 1}$, $\abs{\grn}=\prod_{\ell=1}^dn_\ell$ 
and $\max\grn=\max_{1\leq \ell\leq d}n_\ell$.

\begin{Theorem}\label{thm:loi_des_grands_nombres_orthomartingale}
  Let $\pr{\B,\norm{\cdot}_{\B}}$ be a separable
  $r$-smooth Banach space for some $r\in
(1,2]$,  $1<p<r$ and $d\in\N$. There exists a constant
  $K_{p,d,\B}$ such that the following holds: if
$\pr{X_{\gri}}_{\gr{i}\in\Z^d}$ is an identically distributed
   orthomartingale difference random field such
  that  $\norm{X_{\gr{1}}}_{\B}\in\mathbb L_{p,d-1}$,
  then for all positive $x$, the following inequality holds
  \begin{equation}\label{eq:control_of_sum_PAn}
  \sum_{\gr{n}\in\N^d}\PP\pr{ \abs{\gr{2^{n}}}^{-1/p}
  \max_{\gr{1}\imd\grj\imd \gr{2^n}}\norm{S_{\grj}}_{\B}
  >t}
  \leq K_{p,d}\E{\varphi_{p,d-1}\pr{\frac{\norm{X_{\gr{1}}}_{\B}} x }},
  \end{equation}
  where $S_{\grj}=\sum_{\gr{1}\imd\gri\imd\grj}X_{\gri}$.
  In particular, for some constant $C_{p,d}$ depending only on $p$ and $d$,
  \begin{equation}\label{eq:control_weak_Lp_norm_maximum}
\pr{\sup_{t>0}t^p\PP\pr{\sup_{\gr{n}\smd\gr{1}}\frac{\norm{S_{\gr{n}}}_{\B}
}{\abs{\gr {n } } ^ { 1/p } }>t   }}^{1/p}
  \leq C_{p,d,\B} \norm{X_{\gr{1}}}_{p,d-1}
  \end{equation}
  and the following convergence holds:
  \begin{equation}\label{eq:convergence_presque_sure}
  \lim_{N\to \infty  }\sup_{\grn\smd\gr{1}, 
\max\grn \geq N}\frac{\norm{S_{\gr{n}}}_{\B}
}{\abs{\gr{n}}^{1/p}}
=0\mbox{
  almost surely.}
  \end{equation}
  \end{Theorem}

Note that the condition $X\in\el_{p,d-1}$ cannot be removed, not even
in the independent identically distributed case (see the theorem p.~165 in
\cite{MR346881} for the normalization by $\abs{\grn}$ instead of
$\abs{\grn}^{1/p}$ and \cite{MR494431} for the latter one).

Note that the convergence in \eqref{eq:convergence_presque_sure} holds if only 
one of the coordinates of $\grn$ goes to infinity and uniformly with respect to 
the other coordinates of $\grn$. For example, if $d=2$, we have 
\begin{equation}
 \lim_{N\to\infty}\sup_{n_2\geq 
1}\frac{\norm{S_{N,n_2}}_{\B}}{N^{1/p}n_2^{1/p}}=
 \lim_{N\to\infty}\sup_{n_1\geq 
1}\frac{\norm{S_{n_1,N}}_{\B}}{n_1^{1/p}N^{1/p}}=0.
\end{equation}

We now complete this section by giving results in the spirit  
of those obtained in \cite{MR4046858}.

Let $\pr{X_{\gr{i}}}_{\gr{i}\in\Z^d}$ be an i.i.d. real-valued random field.
Theorem~4.1 in \cite{MR494431} gives the equivalence between the following two assertions for
$\alpha>1/2$ and $p\geq \max\ens{1/\alpha,1}$:
\begin{enumerate}
\item $X_{\gr{1}}$ belongs to $\mathbb L^p\log^{d-1}\mathbb L$;
\item for each positive $\varepsilon$,
\begin{equation}
\sum_{\gr{n}\in\N^d}\abs{\gr{n}}^{p\alpha-2}\PP\pr{\max_{\gr{1}\imd
\gr{i}\imd\gr{n}}
\abs{S_{\gr{i}}}>\varepsilon \abs{\gr{n}}^{\alpha}}<+\infty.
\end{equation}
\end{enumerate}

Deviation inequalities has been used in \cite{MR2794415,MR3451971} for
the question of complete convergence of orthomartingale differences random
fields.

\begin{Theorem}\label{thm:large_deviation_orthomartingale_stationary_p_leq_2}
 Let $\B$ be a separable $r$-smooth Banach space. For each identically
 distributed
$\B$-valued orthomartingale difference
 random field $\pr{X_{\gr{i}}}_{\gr{i}\in \Z^d}$ , for each positive
$\varepsilon$ and each
 $\alpha\in \left(1/r,1\right]$, the following inequality
 takes place:
\begin{equation}\label{eq:Baum_Katz_champs}
\sum_{\gr{n}\in\N^d}\abs{\gr{n}}^{r\alpha-2}\PP\pr{\max_{\gr{1}\imd
\gr{i}\imd\gr{n}}
\norm{S_{\gr{i}}}_{\B}>\varepsilon \abs{\gr{n}}^{\alpha}}\leq C\pr{r,d,\B}
\E{ \varphi_{r,2d}\pr{\frac{\norm{X_{\gr{1}}}_{\B}}{\varepsilon} }   }.
\end{equation}
\end{Theorem}
\begin{Remark}
One could also formulate the corresponding result where $r$ is replaced in
\eqref{eq:Baum_Katz_champs} by $1<p<r$. But this could be established in a more
general context
than ours, namely, that of stochastically dominated orthomartingale differences
random fields,
by using truncation arguments like in \cite{MR2743029}.
\end{Remark}

\begin{Theorem}\label{thm:large_deviation_orthomartingale_stationary_p_geq_2}
 Let $\B$ be a separable $r$-smooth Banach space and $s>r$. For each
identically distributed $\B$-valued orthomartingale
 difference random field $\pr{X_{\gr{i}}}_{\gr{i}\in \Z^d}$, for each positive
$\varepsilon$ and
 each $\alpha\in \left(1/r,1\right]$, the following inequality takes place
\begin{equation}\label{eq:Baum_Katz_champs_pgeq_2}
\sum_{\gr{n}\in\N^d}\abs{\gr{n}}^{s\pr{\alpha-1/r}-1}\PP\pr{\max_{\gr{1}\imd
\gr{i}\imd\gr{n}}
\norm{S_{\gr{i}}}_{\B}>\varepsilon \abs{\gr{n}}^{\alpha}}\\
\leq C\pr{r,d, \B}
\E{ \varphi_{s,d}\pr{\frac{\norm{X_{\gr{1}}}_{\B}}{\varepsilon}}  }.
\end{equation}
\end{Theorem}

\begin{Remark}
On one hand, the results in \cite{MR3451971} require boundedness of
the conditional moments, whereas we do not. On the other hand, their result do not require
that $\pr{\abs{X_{\gr{i}}}}_{\gr{i}\in\Z^d}$ is identically distributed hence
the
results are not directly comparable.
\end{Remark}

\subsection{Law of large number for weighted sums of orthomartingale difference
random fields}

In this subsection, we study the convergence rates in the law of large numbers 
for identically distributed orthomartingale difference random fields. As we 
work with Banach space valued random variables, we
may consider sums of linear bounded operators from a 
Banach space $\B$ to itself. 

\begin{Theorem}\label{thm:law_large_numbers_weighted_sums}
Let $\pr{\B,\norm{\cdot}_{\B}}$ be a separable $r$-smooth 
Banach space for some 
$r\in (1,2]$. For $\gri\in\Z^d$ and $n\geq 1$, let $A_{n,\gri}\colon\B\to\B$ 
be a linear bounded operator and denote its norm by 
$\norm{A_{n,\gri} 
}_{\Bca\pr{\B}}:=\sup\ens{\norm{A_{n,\gri}\pr{x}}_{\B}/\norm{x}_{\B }, 
x\in\B\setminus\ens{0}  
}$. Let $1<p<r$ and let $C_{n,p}:=\pr{\sum_{\gri\in\Z}
\norm{A_{n,\gri} 
}_{\Bca\pr{\B}}^p}^{1/p}$. 
Let $\pr{X_{\gri}}_{\gri\in\Z^d}$ be an identically 
distributed $\B$-valued
orthomartingale difference random field and assume that 
$\norm{X_{\gr{1}}}\in\el^s$ for some $s>p$. Then 
for each positive $\varepsilon$ and each increasing sequence of positive 
numbers $\pr{R_n}_{n\geq 1}$ such that $R_n\to\infty$, 
\begin{equation}\label{eq:law_large_numbers_weighted_sums}
 \sum_{n\geq 1}\pr{R_n^s-R_{n-1}^s}
 \PP\pr{\norm{\sum_{\gri\in\Z^d}A_{n,\gri}\pr{X_{\gri}} }_{\B}>\varepsilon 
C_n^{1/p} R_n}<\infty.
\end{equation}

\end{Theorem}

Let us give an example where Theorem~\ref{thm:law_large_numbers_weighted_sums} 
can be used. Suppose that $\pr{\Lambda_n}_{n\geq 1}$ is  
a sequence of subsets of $\Z^d$ such that 
$\pr{\operatorname{Card}\pr{\Lambda_n}}_{n\geq 1}$ forms an increasing sequence 
(note that we do not assume the sequence $\pr{\Lambda_n}_{n\geq 1}$ to be 
increasing). Let $A_{n,\gri}$ be the identity operator if $\gri$ belongs to 
$\Lambda_n$ and $A_{n,\gri}=0$ otherwise. For a positive 
$\gamma$, define $R_n=\operatorname{Card}\pr{\Lambda_n}^\gamma$. 
Then \eqref{eq:law_large_numbers_weighted_sums} reads
\begin{equation}\label{eq:law_large_numbers_sums_on_sets}
 \sum_{n\geq 
1}\pr{\operatorname{Card}\pr{\Lambda_n}^{s\gamma}-
\operatorname{Card}\pr{\Lambda_{n -1} } ^ { s \gamma}}
 \PP\pr{\norm{\sum_{\gri\in\Lambda_n} X_{\gri} }_{\B}>\varepsilon  
\operatorname{Card}\pr{
\Lambda_n}^{\frac 1p+\gamma}}<\infty.
\end{equation}

\section{Proofs}

\subsection{Proof of Proposition~\ref{prop:exemple_filtra_commutantes}}

As pointed out before, it suffices to prove the second item.
We will use the following lemma:
\begin{Lemma}\label{lem:esperance_conditionnelle_produit}
 Let $\Gca_\ell, 1\leq \ell\leq L$ be independent sub-$\sigma$-algebras 
 of $\Fca$, where $\pr{\Omega,\Fca,\PP}$ is a probability space. For each 
$\ell\in\ens{1,\dots, L}$, let $\Gca'_\ell$ be a sub-$\sigma$-algebra of 
$\Gca_\ell$ and $A_\ell\in\Gca_\ell$. Then 
\begin{equation}\label{eq:esperance_conditionnelle_produit}
 \E{\prod_{\ell=1}^L  \ipr{A_\ell}\mid \bigvee_{\ell=1}^L \Gca'_\ell  }
 =\prod_{\ell=1}^L \E{\ipr{A_\ell} \mid \Gca'_\ell }.
\end{equation}
\end{Lemma}
\begin{proof}
 Note that the random variable in the right hand side of 
\eqref{eq:esperance_conditionnelle_produit} is measurable with respect to 
$\bigvee_{\ell=1}^L \Gca'_\ell$ hence it suffices to show that for each $G
\in \bigvee_{\ell=1}^L \Gca'_\ell$,  
\begin{equation}\label{eq:esperance_conditionnelle_produit_etape}
  \E{\prod_{\ell=1}^L  \ipr{A_\ell} \ipr{G} }
 =\E{\prod_{\ell=1}^L \E{\ipr{A_\ell} \mid \Gca'_\ell } \ipr{G}}.
\end{equation}
Since $\bigvee_{\ell=1}^L \Gca'_\ell$ is generated by the $\pi$-system of sets 
of the form $\bigcap_{\ell=1}^LG_\ell$, $G_\ell\in \Gca'_\ell$, it suffices to 
check \eqref{eq:esperance_conditionnelle_produit_etape} when $G$ has this form, 
that is, 
\begin{equation}\label{eq:esperance_conditionnelle_produit_etape2}
 \forall G_1\in\Gca'_1,\dots,G_L\in\Gca'_L, \quad 
   \E{\prod_{\ell=1}^L  \ipr{A_\ell} \prod_{\ell'=1}^L\ipr{G_{\ell'}} }
   =\E{\prod_{\ell=1}^L \E{\ipr{A_\ell} \mid \Gca'_\ell 
}\prod_{\ell'=1}^L\ipr{G_{\ell'}}}.
\end{equation}
The left hand side of \eqref{eq:esperance_conditionnelle_produit_etape2} 
is $\prod_{\ell=1}^L\PP\pr{A_\ell\cap G_\ell}$ since the sets $A_\ell\cap 
G_\ell$ belong to $\Gca_\ell$. Then note that 
\begin{equation}
 \PP\pr{A_\ell\cap G_\ell}
 =\E{ \ipr{G_\ell}\E{\ipr{A_\ell}\mid\Gca'_\ell  } }
\end{equation}
hence 
\begin{equation}
  \E{\prod_{\ell=1}^L  \ipr{A_\ell} \prod_{\ell'=1}^L\ipr{G_{\ell'}} }
  =\prod_{\ell=1}^L \E{ \ipr{G_\ell}\E{\ipr{A_\ell}\mid\Gca'_\ell  } }
\end{equation}
and an other use of the independence of the $\sigma$-algebras gives 
\eqref{eq:esperance_conditionnelle_produit_etape2} and finishes the proof of 
Lemma~\ref{lem:esperance_conditionnelle_produit}.
\end{proof}

Let $\gri\in\Z^d,\grj\in\Z^d$ of the form 
$\gri=\pr{\gr{i^{\pr{\ell}}}}_{\gr{i^{\pr{\ell}} } \in\Z^{d_\ell}}$ and 
$\grj=\pr{\gr{j^{\pr{\ell}}}}_{\gr{j^{\pr{\ell}} } \in\Z^{d_\ell}}$
Since item 1 of Definition~\ref{dfn:commutante} is clear, it remains to check 
\eqref{eq:def_filtration_commutante}.
 To do so, it suffices to check it when 
$Y$ is $\Fca_{\gri}$-measurable. By a standard approximation argument and  
Dynkin's theorem, it suffices to do it when 
$Y=\prod_{\ell=1}^L\ipr{A_\ell}$, where 
$A_\ell\in \Fca^{\pr{\ell}}_{\gr{i^{\pr{\ell}}}}$. Applying 
Lemma~\ref{lem:esperance_conditionnelle_produit} to these $A_\ell$, 
$\Gca'_\ell=\Fca_{\gr{j^{\pr{\ell}}}}^{\pr{\ell}}$ and 
$\Gca_\ell=\Fca_{\max\ens{\gr{i^{\pr{\ell}}} ,\gr{j^{\pr{\ell}}}} }$, where the 
maximum is taken coordinatewise, we get
\begin{align}
 \E{\prod_{\ell=1}^L\ipr{A_\ell} \mid\Fca_{\grj}}&=\prod_{\ell=1}^L
 \E{\ipr{A_\ell} \mid\Fca_{ \gr{j^{\pr{\ell}} }}^{\pr{\ell}}} \\
 &=\prod_{\ell=1}^L
 \E{\ipr{A_\ell} \mid\Fca_{\min\ens{\gr{i^{\pr{\ell}}} ,\gr{j^{\pr{\ell}}}} }}\\
 &=  \E{\prod_{\ell=1}^L\ipr{A_\ell} \mid \bigvee_{\ell=1}^L 
\Fca_{\min\ens{\gr{i^{\pr{\ell}}} ,\gr{j^{\pr{\ell}}}} }  }
\end{align}
where the second equality uses the commutativity of the filtration 
$\pr{\Fca_{ \gr{i^{\pr{\ell}} }}^{\pr{\ell}}}_{\gr{i^{\pr{\ell}} 
}\in\Z^{d_\ell}    }$ and the third one by an other use of 
Lemma~\ref{lem:esperance_conditionnelle_produit} to this time 
$\Gca'_\ell=\Fca_{\min\ens{\gr{i^{\pr{\ell}}} ,\gr{j^{\pr{\ell}}}} }$. This 
ends the proof of Proposition~\ref{prop:exemple_filtra_commutantes}.

\subsection{Proof of Theorem~\ref{thm:deviation_orthomartingales}}

We will proceed by induction over the dimension $d$. We will actually 
show the following assertion $A\pr{d}$ by induction:
" For each $1<p\leq 2$ and 
each $q>p$, there exists a function $f_{p,q,d}\colon \pr{0,\infty}\to\pr{0,\infty}$ 
such that if $\pr{\B,\norm{\cdot}_{\B}}$ is a separable Banach space for 
which $C_{p,\B}$ defined as in \eqref{eq:definition_constant_Banach_space} is 
finite and $\pr{X_{\gri}}_{\gri\in\Z^d}$ is an orthomartingale difference random field taking values in $\B$, then  
\begin{multline}\label{eq:deviation_orthomartingale_A(d)}
 \PP\pr{\max_{\gr{1}\imd \grn\imd \gr{N}}\norm{\sum_{\gr{1}\imd \gri\imd
 \grn }X_{\gri }  }_{\B} >t
 }
  \leq f_{p,q,d}\pr{C_{p,\B}}
 \int_0^1  u^{q-1} \mathbb
P\pr{
 \pr{\sum_{\gr{1}\imd \gri\imd \gr{N}}\norm{X_{\gri}}_{\B}^p}^{\frac 1p}>tu  }
\mathrm{d}u+\\ +
f_{p,q,d}\pr{C_{p,\B}}
 \int_1^\infty u^{p-1}\pr{1+\log u}^{d-1}\mathbb
P\pr{
 \pr{\sum_{\gr{1}\imd \gri\imd \gr{N}}\norm{X_{\gri}}_{\B}^p}^{\frac 1p}>tu  }
\mathrm{d}u."
\end{multline}
For an $r$-smooth, the constant $C_{p,\B}$ is finite for all $p\in (1,r]$ hence 
$A\pr{d}$ contains the statement of Theorem~\ref{thm:deviation_orthomartingales}.

\subsubsection{The case $d=1$}
The statement of  $A\pr{1}$ is
exactly Proposition~\ref{prop:deviation_martingales_power}, for
which a proof is given right after.

\subsubsection{Induction step}

We will proceed by induction on the dimension $d$. We will denote by $\gri$
the elements of $\Z^d$ and $\pr{\gri,i_{d+1}}$ (and similarly for other
letters) the elements of $\Z^{d+1}$.

We assume that $A\pr{d}$ holds. This
means that for each   $1<p\leq 2$, $q>p$ and
each   separable
Banach space $\pr{\B,\norm{\cdot}_{\B}}$ for which  $C_{p,\B}$ defined as in \eqref{eq:definition_constant_Banach_space} is 
finite, there exists
a function $f_{p,q,d}\colon \pr{0,\infty}\to\pr{0,\infty}$ such that if $\pr{X_{\gri}}_{\gri\in
\Z^d}$ is a $\B$-valued orthomartingale martingale
differences
random field with respect to a completely commuting filtration
$\pr{\Fca_{\gri}}_{\gri\in\Z^d}$, then for each
$1<p\leq r$,
$q>p$ and  $x>0$, \eqref{eq:deviation_orthomartingale_A(d)} holds.

Let $\B$ be such a Banach space and let $1< p\leq 2$ and $q>p$. In order
to complete the induction step we have to find a function
$f_{p,q,d+1}\pr{C_{p,\B}}$ such that if
$\pr{X_{\gri,i_{d+1}}}_{\gri\in\Z^d,i_{d+1}\in\Z}$ is an orthomartingale
differences random field with respect to the completely commuting filtration
$\pr{\Fca_{\gri,i_{d+1}}}_{\gri\in\Z^d,i_{d+1}\in\Z}$, then
\begin{multline}\label{eq:deviation_orthomartingale_dim_d+1}
 \PP\pr{\max_{\gr{1}\imd \grn\imd \gr{N}}\max_{1\leq
n_{d+1}\leq N_{d+1}}\norm{\sum_{\gr{1}\imd \gri\imd
 \grn }\sum_{i_{d+1}=1}^{n_{d+1}}X_{\gri,i_{d+1} }  }_{\B} >t
 }\\
  \leq f_{p,q,d+1}\pr{C_{p,\B}}
 \int_0^\infty \min\ens{u^{q-1},u^{p-1}}\pr{1+\log u}^{d+1-1}\mathbb P\pr{
 \pr{\sum_{\gr{1}\imd \gri\imd
\gr{N}}\sum_{i_{d+1}=1}^{N_{d+1}}\norm{X_{\gri,i_{d+1}}}_{\B}^p}^{1/p}>tu  }
\mathrm{d}u.
\end{multline}
Let $\pr{X_{\gri,i_{d+1}}}_{\gri\in\Z^d,i_{d+1}\in\Z}$ and
$\pr{\Fca_{\gri,i_{d+1}}}_{\gri\in\Z^d,i_{d+1}\in\Z}$ be such a random field
and a filtration. Let $\grN\in\N^d$ be fixed and such that $\grN\smd\gr{1}$.
\begin{equation}
 Y_{n_{d+1}}:=\max_{\gr{1}\imd \grn\imd \gr{N}} \norm{\sum_{\gr{1}\imd \gri\imd
 \grn }\sum_{i_{d+1}=1}^{n_{d+1}}X_{\gri,i_{d+1} }  }_{\B}
\end{equation}
and $\Gca_{n_{d+1}}:=\Fca_{\gr{N},n_{d+1}}$. Then $Y_{n_{d+1}}$ is
$\Gca_{n_{d+1}}$-measurable and
\begin{equation}
 \E{Y_{n_{d+1}}\mid\Gca_{n_{d+1}-1}}
 \geq \max_{\gr{1}\imd \grn\imd \gr{N}} \norm{\sum_{\gr{1}\imd \gri\imd
 \grn }\E{\sum_{i_{d+1}=1}^{n_{d+1}}X_{\gri,i_{d+1} }  \mid\Gca_{n_{d+1}-1}
}}_{\B}
\end{equation}
and the orthomartingale difference property gives that
\begin{equation}
 \E{\sum_{\gr{1}\imd \gri\imd
 \grn } X_{\gri,n_{d+1} }  \mid\Gca_{n_{d+1}-1}}=0
\end{equation}
hence
\begin{equation}
  \E{Y_{n_{d+1}}\mid\Gca_{n_{d+1}-1}}\geq Y_{n_{d+1}};
\end{equation}
in other words, $\pr{Y_{n_{d+1}},\Fca_{n_{d+1}}}$ is a non-negative
sub-martingale. Therefore, by Lemma~\ref{lem:Doob}, we derive that
\begin{multline}\label{eq:deviation_orthomartingale_dim_d+1_step1}
 \tail{\max_{\gr{1}\imd \grn\imd \gr{N}}\max_{1\leq
n_{d+1}\leq N_{d+1}}\norm{\sum_{\gr{1}\imd \gri\imd
 \grn }\sum_{i_{d+1}=1}^{n_{d+1}}X_{\gri,i_{d+1} }  }_{\B}}\pr{t}
 \\
  \leq 4 T_{1,\infty,0}
  \tail{\max_{\gr{1}\imd \grn\imd \gr{N}} \norm{\sum_{\gr{1}\imd \gri\imd
 \grn }\sum_{i_{d+1}=1}^{N_{d+1}}X_{\gri,i_{d+1} }  }_{\B}}\pr{t/4}.
\end{multline}
Then we use the induction assumption $A\pr{d}$ in the setting
$\widetilde{X_{\gri}}=\sum_{i_{d+1}=1}^{N_{d+1}}X_{\gri,i_{d+1} }$,
$\widetilde{\Fca_{\gri}}=\Fca_{\gri,N_{d+1}}$  This gives
 \begin{multline}\label{eq:deviation_orthomartingale_dim_d+1_step2}
 \tail{\max_{\gr{1}\imd \grn\imd \gr{N}}\max_{1\leq
n_{d+1}\leq N_{d+1}}\norm{\sum_{\gr{1}\imd \gri\imd
 \grn }\sum_{i_{d+1}=1}^{n_{d+1}}X_{\gri,i_{d+1} }  }_{\B}}\pr{t
 }\\
  \leq 4f_{p,q,d}\pr{C_{p,\B}}  T_{1,\infty,0}\circ  T_{p,q,d-1}
  \pr{\tail{ \pr{\sum_{\gr{1}\imd\gri\imd\gr{N}}\norm{
\sum_{i_{d+1}=1}^{N_{d+1}}X_{\gri,i_{d+1} } }_{\B}^p}^{1/p}}}\pr{\frac{t}{4}}
.
\end{multline}

The control of the tail of $\pr{\sum_{\gr{1}\imd\gri\imd\gr{N}}\norm{
\sum_{i_{d+1}=1}^{N_{d+1}}X_{\gri,i_{d+1} } }_{\B}^p}^{1/p}$ will be done by an
other use of the martingale property by summing on the $\pr{d+1}$-th
coordinate. More precisely, define the separable Banach space
$\pr{\widetilde{\B},\norm{\cdot}_{\widetilde{\B}}}$ by
\begin{equation}
 \widetilde{\B}=\ens{\pr{x_{\gri}}_{\gr{1}\imd\gri\imd\gr{N}} , x_{\gri}\in \B
\mbox{ for each }\gr{1}\imd\gri\imd\gr{N} }
\end{equation}
\begin{equation}
 \norm{\pr{x_{\gri}}_{\gr{1}\imd\gri\imd\gr{N}} }_{\widetilde{\B}}
 =\pr{\sum_{ \gr{1}\imd\gri\imd\gr{N}   }\norm{x_{\gri}}_{\B}^p    }^{1/p}.
\end{equation}
Observe that if $\pr{D_k}_{k=1}^n$ is a martingale differences sequence taking
its values in $\widetilde{\B}$, then denoting by
$D_{k,\gri},\gr{1}\imd\gri\imd\gr{N} $ the coordinates of $D_k$, we have
\begin{equation}
 \E{\norm{\sum_{k=1}^n D_k}_{\widetilde{\B}}^p}
 =\E{\sum_{\gr{1}\imd\gri\imd\gr{N}}\norm{
\sum_{k=1}^nD_{k,\gri}}_{\B}^p}
\leq C_{p,\B}\E{\sum_{\gr{1}\imd\gri\imd\gr{N}}\sum_{k=1}^n\norm{
D_{k,\gri}}_{\B}^p}
\end{equation}
hence
\begin{equation}
  \E{\norm{\sum_{k=1}^n D_k}_{\widetilde{\B}}^p}\leq
   C_{p,\B}\E{ \sum_{k=1}^n\norm{
D_{k}}_{\widetilde{\B}^p}}.
\end{equation}
This shows, by the definition \eqref{eq:definition_constant_Banach_space} of
$C_{p,\B}$, that $C_{p,\widetilde{\B}}\leq C_{p,\B}$. Moveover, considering
$\widetilde{\B}$-valued martingale
differences sequences which vanish on all the coordinates which are not
$\gr{1}$ shows that $C_{p,\widetilde{\B}}=C_{p,\B}$.

We thus apply the induction assumption to the case $d=1$,  the Banach
$\pr{\widetilde{\B},\norm{\cdot}_{\widetilde{\B}}}$ and the
$\widetilde{\B}$-valued martingale differences sequence
$\pr{D_k}_{k=1}^{N_{d+1}}$  (with respect to
the filtration $\pr{\Fca_{\gr{N},k  }}_{k=0}^{N_{d+1}}$) given by
\begin{equation}
 D_k=\pr{D_{k,\gri}}_{\gr{1}\imd \gri\imd\gr{N}}, D_{k,\gri}
 =X_{\gri,k}.
\end{equation}
Since
\begin{equation}
 \sum_{k=1}^{N_{d+1}}\norm{D_k}_{\widetilde{\B}}^p=
  \sum_{k=1}^{N_{d+1}}\sum_{\gr{1}\imd \gri\imd\gr{N}}
  \norm{X_{\gri,k}}_{\B}^p
\end{equation}
and
\begin{equation}
 \norm{\sum_{k=1}^{N_{d+1}}D_k}_{\widetilde{\B}}^p=
 \sum_{\gr{1}\imd \gri\imd\gr{N}} \norm{\sum_{k=1}^{N_{d+1}}
  X_{\gri,k}}_{\B}^p,
\end{equation}
we get that

\begin{multline}\label{eq:deviation_orthomartingale_dim_d+1_step3}
  \tail{ \pr{\sum_{\gr{1}\imd\gri\imd\gr{N}}\norm{
\sum_{i_{d+1}=1}^{N_{d+1}}X_{\gri,i_{d+1} } }^p}^{1/p}}\pr{t/4}
=\tail{\norm{\sum_{k=1}^{N_{d+1}}D_k}_{\widetilde{\B}}}\pr{t/4}\\
\leq f_{p,q,1}\pr{C_{p,\B}}
T_{p,q,0}\pr{\tail{
 \pr{ \sum_{k=1}^{N_{d+1}}\norm{D_k}_{\widetilde{\B}}^p}^{1/p}   }}\pr{t/4}\\
=f_{p,q,1}\pr{C_{p,\B}}
T_{p,q,0}\pr{\tail{ \pr{\sum_{k=1}^{N_{d+1}}\sum_{\gr{1}\imd \gri\imd\gr{N}}
  \norm{X_{\gri,k}}_{\B}^p}^{1/p}  }  }\pr{t/4}.
\end{multline}
The combination \eqref{eq:deviation_orthomartingale_dim_d+1_step2} with
\eqref{eq:deviation_orthomartingale_dim_d+1_step3} gives the bound
 \begin{multline}\label{eq:deviation_orthomartingale_dim_d+1_step4}
 \tail{ \max_{\gr{1}\imd \grn\imd \gr{N}}\max_{1\leq
n_{d+1}\leq N_{d+1}}\norm{\sum_{\gr{1}\imd \gri\imd
 \grn }\sum_{i_{d+1}=1}^{n_{d+1}}X_{\gri,i_{d+1} }  }_{\B}}\pr{t
 }\\
  \leq  4f_{p,q,d}\pr{C_{p,\B}} f_{p,q,1}\pr{C_{p,\B}} T_{p,q,0}\circ
T_{1,\infty,0}\circ  T_{p,q,d-1}
 \tail{\pr{\sum_{\gr{1}\imd \gri\imd\gr{N}}\sum_{i_{d+1}=1}^{N_{d+1}}
  \norm{X_{\gri,i_{d+1}}}_{\B}^p}^{1/p}}\pr{t/4}.
\end{multline}
 Using twice Lemma~\ref{lem:operateur_pq} ends the proof of
Theorem~\ref{thm:deviation_orthomartingales}.

\subsection{Proof of Corollary~\ref{cor:dev_ortho_stoch_dom}}

We have to bound the tail of
$\sum_{\gr{1}\imd\gri\imd\gr{N}}\norm{X_{\gri}}_{\B}^p$ in terms of the tail
of $V$. To do so,  we apply successively
Theorem~\ref{thm:deviation_orthomartingales},
\eqref{eq:T_pq_tail_powers}, \eqref{eq:conv_ordering_tails}
and Lemma~\ref{lem:operateur_pq} in order to derive
\begin{align}
 \tail{\max_{\gr{1}\imd \grn\imd \gr{N}}\norm{\sum_{\gr{1}\imd \gri\imd
 \grn }X_{\gri }  }_{\B} }\pr{t } 
  &\leq K\pr{p,q,d,\B}T_{p,q,d-1}\pr{\tail{\pr{\sum_{\gr{1}\imd
  \gri\imd\gr{N}} \norm{X_{\gri}}_{\B}^p}^{1/p}  }  }\pr{t}\\
  &=K\pr{p,q,d,\B}p^{-1}T_{1,q/p,d-1}\pr{\tail{ \sum_{\gr{1}\imd
  \gri\imd\gr{N}} \norm{X_{\gri}}_{\B}^p}   }  \pr{t^p}\\
  &\leq K'\pr{p,q,d,\B}T_{1,q/p,d-1}\circ T_{1,\infty,0}\pr{\tail{
 V^p}   }  \pr{t^p}\\
 &\leq
 K''\pr{p,q,d,\B}T_{1,q/p,d-1+1} \pr{\tail{
 V^p}   }  \pr{t^p}
\end{align}
 and we apply another time \eqref{eq:T_pq_tail_powers}.
 This concludes the proof of
Corollary~\ref{cor:dev_ortho_stoch_dom}.

\subsection{Proof of Theorem~\ref{thm:LLN_regression}}

Let $\B=\mathbb L^p\pr{[0,1]^d,\lambda_d}$ endowed with
the norm $\norm{f}_{\B}=\pr{\int_{[0,1]^d}\abs{f\pr{x}}^p
dx}^{1/p}$. Let $n\geq 1$ be fixed.
Define for $ \gr{i}\in\Lambda_n$ the $\B$-valued random variable
$\widetilde{X_{\gri}}$ by
\begin{equation}
 \widetilde{X_{\gri}}:=\gr{x}\in [0,1]^d\mapsto
 \frac{ X_{\gr{i}}
 K\pr{\frac{\gr{x}-\gr{i}/n}{h_n}}}
 {\sum_{\gr{j}\in\Lambda_n} K\pr{\frac{\gr{x}-\gr{j}/n}{h_n}}}
 ,\quad \gr{x}\in [0,1]^d.
\end{equation}
Then $g_n\pr{\cdot}-\E{g_n\pr{\cdot}}=\sum_{
\gri\in\Lambda_n}\widetilde{X_{\gri}}$. We will now show that 
for $p'=\min\ens{p,2}$, 
\begin{equation}
 \sum_{\gri\in\Lambda_n}
 \norm{ \widetilde{X_{\gri}}}_{\B}^{p'}\conv 
 \begin{cases}
\kappa\pr{K,p}\abs{X_{\gr{1}}}^{p}\pr{nh_n}^{d\pr{1-p}}& \mbox{ if }1<p\leq 2.\\
\kappa\pr{K,p}X_{\gr{1}}^2\pr{nh_n}^{d\pr{1-p}} n^{\frac{p-2}p}&
\mbox{ if }p>2,
 \end{cases}
 \end{equation}
where $\kappa\pr{K,p}$ depends only on $K$ and $p$.
Define 
\begin{equation}
 a_{\gri,p}:= \pr{\int_{[0,1]^d} \frac{
K\pr{\frac{\gr{x}-\gr{i}/n}{h_n}}^p}  { \pr{\sum_{\gr{j}\in\Lambda_n}
K\pr{\frac{\gr{x}-\gr{j}/n}{h_n}}  }^{p}   }  dx  }^{p'/p}
\end{equation}
 and observe that
\begin{equation*}
  \sum_{\gri\in\Lambda_n}
 \norm{ \widetilde{X_{\gri}}}_{\B}^{p'}=
  \sum_{\gri\in\Lambda_n}\abs{X_{\gri}}^{p'}\alpha_{\gri,p}.
\end{equation*}

Let $\varphi\colon\R\to\R$ be a convex increasing function. Then
denoting $A_p=\sum_{\gri\in\Lambda_n} \alpha_{\gri,p}$,
convexity and identical distribution of the $X_{\gri}$ imply that
\begin{align}
 \E{\varphi\pr{
  \sum_{\gri\in\Lambda_n}
 \norm{ \widetilde{X_{\gri}}}_{\B}^{p'}}}
& = \E{\varphi\pr{
  \sum_{\gri\in\Lambda_n}\frac{\alpha_{\gri,p}}{A_p} A_p\abs{X_{\gri}}^{p'}
}} \\
&\leq  \sum_{\gri\in\Lambda_n}\frac{\alpha_{\gri,p}}{A_p}\E{\varphi\pr{
   A_p\abs{X_{\gri}}^{p'}
}}  \\
&=\sum_{\gri\in\Lambda_n}\frac{\alpha_{\gri,p}}{A_p}\E{\varphi\pr{
   A_p\abs{X_{\gr{1}}}^{p'}
}}
\end{align}
hence
\begin{equation}
 \sum_{\gri\in\Lambda_n}
 \norm{ \widetilde{X_{\gri}}}_{\B}^{p'}\conv
A_p\abs{X_{\gr{1}}}^{p'}   .
\end{equation}
It remains to bound the term
$A_p=\sum_{\gri\in\Lambda_n}\alpha_{\gri,p}$. First assume that $p\leq 2$. Then
$p'=p$ and
\begin{equation}
 \sum_{\gri\in\Lambda_n}\alpha_{\gri,p}
 =  \int_{[0,1]^d} \frac{\sum_{\gri\in\Lambda_n}
K\pr{\frac{\gr{x}-\gr{i}/n}{h_n}}^p}  { \pr{\sum_{\gr{j}\in\Lambda_n}
K\pr{\frac{\gr{x}-\gr{j}/n}{h_n}}  }^{p}   }  dx.
\end{equation}
Since $K$ is supported on $[-1,1]^d$, by assumptions~\ref{itm:assumption1} and
\ref{itm:assumption_bounds}, we derive that
\begin{equation}\label{eq:bound_kernel}
 c^s \ind{n\gr{x}-nh_n\gr{1} \imd \gri \imd
n\gr{x}+nh_n\gr{1}} \leq K\pr{\frac{\gr{x}-\gr{i}/n}{h_n}}^s\leq
C^s\ind{n\gr{x}-nh_n\gr{1} \imd \gri \imd n\gr{x}+nh_n\gr{1}},s\in\ens{1,p}
\end{equation}
 hence there exists a constant $\kappa$ such that
 \begin{equation}
  \sum_{\gri\in\Lambda_n}\alpha_{\gri,p}\leq \kappa
  \pr{nh_n}^{d\pr{1-p}}.
 \end{equation}
Now assume that $p>2$. An application of Hölder's inequality
to the conjugate exponents $p/2$ and $p/\pr{p-2}$ gives
\begin{equation}
 \sum_{\gri\in\Lambda_n}\alpha_{\gri,p}\leq
  \sum_{\gri\in\Lambda_n}\alpha_{\gri,p}^{p/2}
\operatorname{Card}\pr{\Lambda_n}^{ \pr{p-2}/p}.
\end{equation}
Using again \eqref{eq:bound_kernel} gives the bound
 \begin{equation}
  \sum_{\gri\in\Lambda_n}\alpha_{\gri,p}^{p/2} \leq \kappa \pr{nh_n}^{d\pr{1-p}}
 \end{equation}
hence
\begin{equation}
 \sum_{\gri\in\Lambda_n}\alpha_{\gri,p}\leq \kappa
 \pr{nh_n}^{d\pr{1-p}}n^{d\frac{p-2}p}
\end{equation}
Then an application of Corollary~\ref{cor:dev_ortho_stoch_dom} with $p$ replaced by $p'=\min\ens{p,2}$ allows to conclude.

\subsection{Proof of the results of Subsection~\ref{subsec:LGN}}
\begin{proof}[Proof of Theorem~\ref{thm:loi_des_grands_nombres_orthomartingale}]
  Let us prove \eqref{eq:control_of_sum_PAn} for $x=1$; the general case follow
by
  applying the previous one to $X_{\gr{i}}/x$.
  We define for all $\gr{n}\smd \gr{1}$ the event
 \begin{equation}
 A_{\gr{n}}:=\ens{\abs{\gr{2^{n}}}^{-1/p}
  \max_{\gr{1}\imd\gr{i}\imd \gr{2^n}}\norm{S_{\gr{i}}}_{\B}
  >2}.
\end{equation}
Let us fix $\gr{n}\smd \gr{1}$ and define for $\gr{1}\imd\gr{i}\imd \gr{2^n}$
the
random variables
\begin{equation}
X'_{\gr{i}}:=\sum_{\delta\in\ens{0;1}^d}\pr{-1}^{\delta_1+
\dots+\delta_d}  \E{X_{\gri}\mathbf 1\ens{\norm{X_{\gri}}_{\B}
\leq \abs{\gr{2^n}}^{1/p}
 }  \mid \Fca_{\gri-\delta } } \mbox{ and }
\end{equation}
\begin{equation}
X''_{\gr{i}}:=\sum_{\delta\in\ens{0;1}^d}\pr{-1}^{\delta_1+
\dots+\delta_d}  \E{X_{\gr{i}}\mathbf 1\ens{\norm{X_{\gr{i}}}_{\B}
> \abs{\gr{2^n}}^{1/p}
 }  \mid \Fca_{\gr{i}-\delta } }.
\end{equation}
We denote by $S'_{\gr{n}}$ and $S''_{\gr{n}}$ the respective partial sums.
Since $\pr{X_{\gri}}_{\gr{i}\in\Z^d}$ is an
orthomartingale difference random field
with respect to the filtration $\pr{\Fca_{\gr{i}}}_{\gr{i}\in\Z^d}$, the
equality $X_{\gr{i}}=X'_{\gr{i}}+X''_{\gr{i}}$ holds
hence $\PP\pr{A_{\gr{n}}}\leq \PP\pr{A'_{\gr{n}}}+\PP\pr{A''_{\gr{n}}}$, where
\begin{equation}
A'_{\gr{n}}:=\ens{\abs{\gr{2^{n}}}^{-1/p}
  \max_{\gr{1}\imd\gr{i}\imd \gr{2^n}}\norm{S'_{\gr{i}}}_{\B}
  >1}\mbox{ and }A''_{\gr{n}}:=\ens{\abs{\gr{2^{n}}}^{-1/p}
  \max_{\gr{1}\imd\gr{i}\imd \gr{2^n}}\norm{S''_{\gr{i}}}_{\B}
  >1}.
\end{equation}
By Chebyshev's inequality,
\begin{equation}
 \PP\pr{A'_{\gr{n}}}\leq \abs{\gr{2^{n}}}^{-r/p}
 \E{ \max_{\gr{1}\imd\gr{i}\imd \gr{2^n}}\norm{S'_{\gr{i}}}_{\B}^r},
\end{equation}
and since $\pr{X'_{\gri} }_{\gr{i}\in\Z^d}$ is an orthomartingale
difference random field
with respect to the filtration $\pr{\Fca_{\gr{i}}}_{\gr{i}\in\Z^d}$, Doob's
inequality
and  Proposition\ref{prop:moments_ordre_r_orthomartingale_Banach} gives
\begin{equation}
 \PP\pr{A'_{\gr{n}}}\leq C_{p,d,\B}\abs{\gr{2^{n}}}^{-r/p}
 \sum_{\gr{1}\imd\gr{i}\imd \gr{2^n}}\E{\norm{X'_{\gr{i}}}_{\B}^r}.
\end{equation}
Since
\begin{equation}
\E{\norm{X'_{\gr{i}}}_{\B}^r}=\norm{\norm{\norm{X'_{\gr{i}}}_{\B}}  }_{r}^r
 \leq \kappa\pr{d,r}
 \norm{X_{\gr{i}}\mathbf 1\ens{\norm{X_{\gr{i}}}_{\B}\leq
\abs{\gr{2^n}}^{1/p}}}_r^r,
 \end{equation}
we derive that
\begin{equation}
 \PP\pr{A'_{\gr{n}}}\leq  C_{p,d,\B}\abs{\gr{2^{n}}}^{-r/p}
 \sum_{\gr{1}\imd\gr{i}\imd \gr{2^n}}\E{\norm{X_{\gr{i}}}_{\B}^r\mathbf
1\ens{\norm{X_{\gr{i}}}_{\B}\leq\abs{\gr{2^n}}^{1/p}}}.
\end{equation}
Moreover, since the random variable $X_{\gri}$ has the same distribution as
$X_{\gr{1}}$, we derive that
\begin{equation}
 \E{\norm{X_{\gr{i}}}_{\B}^r\mathbf
1\ens{\norm{X_{\gr{i}}}_{\B}\leq\abs{\gr{2^n}}^{1/p}}}
=\E{\norm{X_{\gr{1}}}_{\B}^r\mathbf
1\ens{\norm{X_{\gr{1}}}_{\B}\leq\abs{\gr{2^n}}^{1/p}}}
\end{equation}

which leads to the bound
\begin{equation}\label{eq:bound_PAn_prim}
\PP\pr{A'_{\gr{n}}}\leq  C_{p,d,\B}\abs{\gr{2^{n}}}^{1-r/p}
\E{\norm{X_{\gr{1}}}_{\B}^r\mathbf
1\ens{\norm{X_{\gr{1}}}_{\B}\leq\abs{\gr{2^n}}^{1/p}}}.
\end{equation}
In order to bound $\PP\pr{A''_{\gr{n}}}$, we use Markov's inequality and
$\max_{\gr{1}\imd\gr{i}\imd \gr{2^n}}\norm{S''_{\gr{i}}}_{\B}\leq
\sum_{\gr{1}\imd\gr{i}\imd \gr{2^n}}\norm{X''_{\gr{i}}}_{\B}$ to get
\begin{equation}
\PP\pr{A''_{\gr{n}}}\leq \abs{\gr{2^{n}}}^{-1/p}
\sum_{\gr{1}\imd\gr{i}\imd \gr{2^n}}\E{\norm{X''_{\gr{i}}}_{\B} }
\leq \abs{\gr{2^{n}}}^{-1/p}
\sum_{\gr{1}\imd\gr{i}\imd \gr{2^n}}\E{\norm{X_{\gr{i}}}_{\B}
\mathbf 1\ens{\norm{X_{\gr{i}}}_{\B}> \abs{\gr{2^{n}}}^{1/p}  }
}.
\end{equation}
An other use of the fact that $ X_{\gri}$ has the same distribution as
$X_{\gr{1}}$ leads to
\begin{equation}\label{eq:bound_PAn_sec}
\PP\pr{A''_{\gr{n}}}\leq 2^d \abs{\gr{2^{n}}}^{1-1/p}
\E{\norm{X_{\gr{1}}}_{\B}
\mathbf 1\ens{\norm{X_{\gr{1}}}_{\B}> \abs{\gr{2^{n}}}^{1/p}  }
}.
\end{equation}
Combining \eqref{eq:bound_PAn_prim} with \eqref{eq:bound_PAn_sec}, we obtain
\begin{multline}
\sum_{\gr{n}\in\N^d}\PP\pr{ \abs{\gr{2^{n}}}^{-1/p}
  \max_{\gr{1}\imd\gr{i}\imd \gr{2^n}}\norm{S_{\gr{i}}}_{\B}
  >2}\leq C_{p,d,\B}\sum_{\gr{n}\in\N^d} \abs{\gr{2^{n}}}^{1-r/p}
\E{\norm{X_{\gr{1}}}_{\B}^r\mathbf
1\ens{\norm{X_{\gr{1}}}_{\B}\leq\abs{\gr{2^n}}^{1/p}}}\\
+ C_{p,d,\B}\sum_{\gr{n}\in\N^d} \abs{\gr{2^{n}}} \PP\pr{
\norm{X_{\gr{1}}}_{\B}  >\abs{\gr{2^n}}^{1/p}
 }
+ C_{p,d,\B} \sum_{\gr{n}\in\N^d}\abs{\gr{2^{n}}}^{1-1/p}
\E{\norm{X_{\gr{1}}}_{\B}\mathbf 1\ens{\norm{X_{\gr{1}}}_{\B} >
\abs{\gr{2^{n}}}^{1/p}  }}.
\end{multline}
The number of elements of $\N^d$ whose sum is $k$ does not exceed
$\pr{k+1}^{d-1}$ hence
\begin{multline}
\sum_{\gr{n}\in\N^d}\PP\pr{ \abs{\gr{2^{n}}}^{-1/p}
  \max_{\gr{1}\imd\gr{i}\imd \gr{2^n}}\norm{S_{\gr{i}}}_{\B}
  >2}\leq  C_{p,d,\B} \sum_{k=0}^{+\infty} 2^{k\pr{1-r/p}}\pr{k+1}^{d-1}
\E{\norm{X_{\gr{1}}}_{\B}^r\mathbf 1\ens{X\leq 2^{k/p}}}\\
+  C_{p,d,\B} \sum_{k=0}^{+\infty} 2^{k} \pr{k+1}^{d-1}\PP\pr{
\norm{X_{\gr{1}}}_{\B}  >2^{k/p}   }
+  C_{p,d,\B}  \sum_{k=0}^{+\infty}2^{k\pr{1-1/p}}\pr{k+1}^{d-1}
\E{\norm{X_{\gr{1}}}_{\B}\mathbf 1\ens{\norm{X_{\gr{1}}}_{\B}> 2^{k/p}  }}.
\end{multline}
Now, \eqref{eq:control_of_sum_PAn} follows from 
\eqref{eq:sum_indicators_2skkd_leq} and \eqref{eq:sum_indicators_2skkd}.

In order to prove \eqref{eq:control_weak_Lp_norm_maximum}, observe that
\eqref{eq:control_of_sum_PAn} entails that for any positive $t$,
\begin{equation}
\PP\pr{\sup_{\gr{n}\smd\gr{1}}\frac{\norm{S_{\gr{n}}}_{\B}
}{\abs{\gr{n}}^{1/p}}
>t}\leq C_{p,d,\B} \E{\varphi_{p,d-1}\pr{\frac{\norm{X_{\gr{1}}}_{\B}}t}},
\end{equation}
where $C_{p,d}$ depends only on $d$ and $p$ and is bigger than $1$.
Let $R$ be a positive number. If $tleq R$, then
$t^p\PP\pr{\sup_{\gr{n}\smd\gr{1}}\frac{\norm{S_{\gr{n}}}_{\B}}{\abs{\gr{n}}^{
1/p}}
>t}\leq R^p$ and if $t>R$, then
\begin{multline*}
t^p\PP\pr{\sup_{\gr{n}\smd\gr{1}}\frac{\norm{S_{\gr{n}}}_{\B}}{\abs{\gr{n}}^{
1/p}}
>t}
\leq
C_{p,d}\E{\norm{X_{\gr{1}}}_{\B}^p\pr{1+L\pr{\frac{\norm{X_{\gr{1}}}_{\B}}t}
}^{d-1}   }\\ \leq
C_{p,d}\E{\norm{X_{\gr{1}}}_{\B}^p\pr{1+L\pr{\frac{\norm{X_{\gr{1}}}_{\B}}R}
}^{d-1}   }
\end{multline*}
hence for all $t$ and all positive $R$,
\begin{multline}
t^p\PP\pr{\sup_{\gr{n}\smd\gr{1}}\frac{\norm{S_{\gr{n}}}_{\B}}{\abs{\gr{n}}^{
1/p}}
>t} \leq R^p+
C_{p,d}\E{\norm{X_{\gr{1}}}_{\B}^p\pr{ L\pr{\frac{\norm{X_{\gr{1}}}_{\B}}R}
}^{d-1}   }\\
\leq C_{p,d}R^p
\E{\varphi_{p,d-1}\pr{\frac{\norm{X_{\gr{1}}}_{\B}}R}}.
\end{multline}
In particular, for all $R>\norm{X}_{p,d-1}$,
\begin{equation}
\sup_{t>0}t^p\PP\pr{\sup_{\gr{n}\smd\gr{1}}\frac{\norm{S_{\gr{n}}}_{\B}}{\abs{
\gr{n}} ^ { 1/p } }
>t}
\leq C_{p,d}R^p,
\end{equation}
which gives \eqref{eq:control_weak_Lp_norm_maximum}.

In order to prove \eqref{eq:convergence_presque_sure}, we define for $1\leq
j\leq d$
the random variable
\begin{equation}
M_{j,p, N}:=\sup\ens{\frac{\norm{S_{\gr{n}}}_{\B}}{\abs{\gr{n}}^{1/p}}
,\gr{n}\smd \gr{1}, n_j\geq 2^N}.
\end{equation}
Then the combination of \eqref{eq:control_of_sum_PAn} with the Borel-Cantelli
lemma
gives that $M_{j,p, N}\to 0$ almost surely. This ends the proof of
Theorem~\ref{thm:loi_des_grands_nombres_orthomartingale}.
  \end{proof}
  
\begin{proof}[Proof of 
Theorem~\ref{thm:large_deviation_orthomartingale_stationary_p_leq_2}]
 
In what follows, $C\pr{r,d,\B}$ will denote a constant 
that depends only on $r$, $d$ and $\B$ and that may 
change from line to line. Observe that partitionning 
$\pr{\N\setminus\ens{0}}^d$ 
into rectangles of the form $\ens{\grn  \in\N^d, \mbox{ for each 
}\ell\in\ens{1,\dots,d},
2^{N_\ell}\leq n_\ell\leq 2^{N_{\ell}+1}-1}$, it suffices 
to prove that 
\begin{equation}\label{eq:Baum_Katz_champs_dyad}
\sum_{\gr{N}\in\N^d}\abs{\gr{2^N}}^{r\alpha-1}\PP\pr{\max_{\gr{1}\imd
\gr{i}\imd \gr{2^N}}
\norm{S_{\gr{i}}}_{\B}>\varepsilon \abs{\gr{2^N}}^{\alpha}}\leq C\pr{r,d,\B}
\E{ \varphi_{r,2d}\pr{\frac{\norm{X_{\gr{1}}}_{\B}}{\varepsilon} }   }.
\end{equation}
Moreover, replacing $\norm{X_{\gr{1}}}_{\B}$ by
$\norm{X_{\gr{1}}}_{\B}/\eps$ if necessary, we can assume that $\eps=1$.
Noticing that $\sum_{\gr{1}\imd\gri \imd\gr{2^N}} 
\norm{X_{\gri}}_{\B}^r\conv \abs{\gr{2^N}}\norm{X_{\gr{1}}}_{\B}$, an 
application of Corollary~\ref{cor:dev_ortho_stoch_dom} with $p=r$, $q=r+1$ and
$V^r=\abs{\gr{2^N}}\norm{X_{\gr{1}}}_{\B}^r$ gives
\begin{multline}\label{eq:demo_pour_r_LGN}
\PP\pr{\max_{\gr{1}\imd
\gr{i}\imd \gr{2^N}}
\norm{S_{\gr{i}}}_{\B}>  \abs{\gr{2^n}}^{\alpha}} \leq
C\pr{r,d,\B}
 \int_0^1 u^{r}  \PP\pr{\norm{X_{\gr1}}_{\B}
>u\abs{\gr{2^N}}^{\alpha-1/r}}
\mathrm{d}u\\
+C\pr{r,d,\B}
 \int_1^\infty  u^{r-1}\pr{1+ \log u}^{d}\PP\pr{\norm{X_{\gr1}}_{\B}
>u\abs{\gr{2^N}}^{\alpha-1/r}}
\mathrm{d}u.
\end{multline}
Using the fact that for a fixed $k$, 
\begin{equation}\label{eq:card_sum_=k}
 \operatorname{Card}\pr{\ens{\gr{N}=\pr{N_\ell}_{\ell=1}^d\in\N^d,
\sum_{\ell=1}^d N_\ell=k}} \leq c_dk^{d-1},
\end{equation}
 we derive that 
\begin{multline}
 \sum_{\gr{N}\in\N^d}\abs{\gr{2^N}}^{r\alpha-1}\PP\pr{\max_{\gr{1}\imd
\gr{i}\imd \gr{2^N}}
\norm{S_{\gr{i}}}_{\B}>\abs{\gr{2^N}}^{\alpha}}
\\ 
\leq 
C\pr{r,d,\B}\sum_{k\geq 1}2^{k\pr{r\alpha-1}}k^{d-1}
 \int_0^1 u \PP\pr{\norm{X_{\gr1}}_{\B}
>u2^{k\pr{\alpha-1/r}} }
\mathrm{d}u\\
+C\pr{r,d,\B}
 \sum_{k\geq 1}2^{k\pr{r\alpha-1}}k^{d-1}\int_1^\infty  u^{r-1}\pr{1+ \log
u}^{d}\PP\pr{\norm{X_{\gr1}}_{\B} 
>u2^{k\pr{\alpha-1/r}} }
\mathrm{d}u.
\end{multline}
Using \eqref{eq:sum_indicators_2skkd}, we get that 
\begin{multline}\label{eq:step_Baum_Katz_r}
 \sum_{\gr{N}\in\N^d}\abs{\gr{2^n}}^{r\alpha-1}\PP\pr{\max_{\gr{1}\imd
\gr{i}\imd \gr{2^N}}
\norm{S_{\gr{i}}}_{\B}>\varepsilon \abs{\gr{2^n}}^{\alpha}}
\\ 
\leq 
C\pr{r,d,\B}
\E{\norm{X_{\gr1}}_{\B}^{r} \int_0^1 u
L\pr{ \frac{ \norm{X_{\gr1}}_{\B}}{u}  }^{d-1}
\ind{\norm{X_{\gr1}}_{\B}>u    }   }
\mathrm{d}u\\
+C\pr{r,d,\B}\E{\norm{X_{\gr1}}_{\B}^{r}\int_1^\infty  u^{ -1}\pr{1+ \log
u}^{d}
L\pr{ \frac{ \norm{X_{\gr1}}_{\B}}{u}  }^{d-1}
\ind{\norm{X_{\gr1}}_{\B}>u    }   }
\mathrm{d}u.
\end{multline}
Using that for a fixed $x$,
\begin{align*}
 \int_0^1u
L\pr{ \frac{ x}{u}  }^{d-1}
\ind{x>u    }
\mathrm{d}u&=\ind{x>1    }\int_0^1u
L\pr{ \frac{ x}{u}  }^{d-1}
\mathrm{d}u+ \ind{x\leq 1    }\int_0^xu
L\pr{ \frac{ x}{u}  }^{d-1}
\mathrm{d}u\\
& =\pr{1+L\pr{x}}^{d-1}
\end{align*}
and using $L\pr{x/u}^{d-1}\leq \pr{L\pr{x}+L\pr{u}}^{d-1
}\leq 2^{d}\pr{L\pr{x}+L\pr{u}}$, we get
\begin{align*}
 \int_1^{\infty}u^{ -1}\pr{1+ \log
u}^{d}
L\pr{ \frac{ x}{u}  }^{d-1}\ind{x>u}\mathrm{d}u&
= \ind{x>1}\int_1^{x}u^{ -1}\pr{1+ \log
u}^{d}
L\pr{ \frac{ x}{u}  }^{d-1}\mathrm{d}u\\
&\leq c_d\ind{x>1}\pr{1+\log x}^{2d},
\end{align*}
which implies that both terms in the right hand side of  \eqref{eq:step_Baum_Katz_r} are smaller than
$K_{r,d,\B}\E{\varphi_{r,2d}\pr{\norm{X_{\gr{1}}}_{\B} }}$.
This ends the proof of Theorem~\ref{thm:large_deviation_orthomartingale_stationary_p_leq_2}.
\end{proof}

  \begin{proof}[Proof of 
Theorem~\ref{thm:large_deviation_orthomartingale_stationary_p_geq_2}]
 
In what follows, $C\pr{r,d,\B}$ will denote a constant that 
depends only on $r$, $d$ and $\B$ and that may change from line to line.

By the same arguments as at the beginning of the proof
of Theorem~\ref{thm:large_deviation_orthomartingale_stationary_p_leq_2},
we are reduced to prove that
 \begin{equation}\label{eq:Baum_Katz_champs_dyad_s>r}
\sum_{\gr{N}\in\N^d}\abs{\gr{2^N}}^{\alpha-1/r}\PP\pr{\max_{\gr{1}\imd
\gr{i}\imd \gr{2^N}}
\norm{S_{\gr{i}}}_{\B}> \abs{\gr{2^N}}^{\alpha}}\leq C\pr{r,d,\B}
\E{ \varphi_{r,d-1}\pr{ \norm{X_{\gr{1}}}_{\B}  }   }.
\end{equation}
To do so, we apply Corollary~\ref{cor:dev_ortho_stoch_dom} 
in the setting $p=r$, $q=s+1$ and $V^p=V^r=\abs{\gr{2^{N}}}
\norm{X_{\gr{1}}}_{\B}^r$. Using again \eqref{eq:card_sum_=k}, we obtain 
\begin{multline}
 \sum_{\gr{N}\in\N^d}\abs{\gr{2^N}}^{\alpha-1/r}\PP\pr{\max_{\gr{1}\imd
\gr{i}\imd \gr{2^N}}
\norm{S_{\gr{i}}}_{\B}>\varepsilon \abs{\gr{2^N}}^{\alpha}}\\
\leq C\pr{r,d,\B}\sum_{k\geq 1}
2^{k\pr{\alpha-1/r}}k^{d-1}\int_0^1
u^s\PP\pr{\norm{X_{\gr1}}_{\B}>u2^{k\pr{\alpha-/r}}  }\D{u}\\
+C\pr{r,d,\B}\sum_{k\geq 1}2^{k\pr{\alpha-1/r}}k^{d-1}\int_1^\infty
u^{r-1}\pr{1+\log u}^d\PP\pr{\norm{X_{\gr1}}_{\B}>u2^{k\pr{\alpha-/r}}  }\D{u}.
\end{multline}
Then using \eqref{eq:sum_indicators_2skkd}, we derive that 
\begin{multline}
 \sum_{\gr{N}\in\N^d}\abs{\gr{2^N}}^{\alpha-1/r}\PP\pr{\max_{\gr{1}\imd
\gr{i}\imd \gr{2^N}}
\norm{S_{\gr{i}}}_{\B}>  \abs{\gr{2^N}}^{\alpha}}\\
\leq C\pr{r,d,\B}\int_0^1\E{\norm{X_{\gr1}}_{\B}^s
L\pr{\frac{\norm{X_{\gr1}}_{\B}}u }^{d-1}\ind{\norm{X_{\gr1}}_{\B}>u} \D{u} }\\
+C\pr{r,d,\B}\int_1^\infty u^{r-s-1}\pr{1+\log u}^d
\E{\norm{X_{\gr1}}_{\B}^s
L\pr{\frac{\norm{X_{\gr1}}_{\B}}u }^{d-1}
\ind{\norm{X_{\gr1}}_{\B}>u} \D{u} }.
\end{multline}
Using that for a fixed $x$,
\begin{align*}
 \int_0^1 
L\pr{ \frac{ x}{u}  }^{d-1}
\ind{x>u    }
\mathrm{d}u&=\ind{x>1    }\int_0^1
L\pr{ \frac{ x}{u}  }^{d-1}
\mathrm{d}u+ \ind{x\leq 1    }\int_0^x
L\pr{ \frac{ x}{u}  }^{d-1}
\mathrm{d}u\\
& =\pr{1+L\pr{x}}^{d-1}
\end{align*}
and using $L\pr{x/u}^{d-1}\leq \pr{L\pr{x}+L\pr{u}}^{d-1
}\leq 2^{d}\pr{L\pr{x}+L\pr{u}}^{d-1}$, we get
\begin{align*}
 \int_1^{\infty}u^{ r-s-1}\pr{1+ \log
u}^{d}
L\pr{ \frac{ x}{u}  }^{d-1}\ind{x>u}\mathrm{d}u&
= \ind{x>1}\int_1^{x}u^{ r-s-1}\pr{1+ \log
u}^{d}
L\pr{ \frac{ x}{u}  }^{d-1}\mathrm{d}u\\
&\leq c_d\ind{x>1}\pr{1+\log x}^{d-1}\int_1^{\infty}u^{ r-s-1}\pr{1+ \log
u}^{2d-1},
\end{align*}
we get \eqref{eq:Baum_Katz_champs_dyad_s>r}, which finishes the proof of 
Theorem~\ref{thm:large_deviation_orthomartingale_stationary_p_geq_2}.
\end{proof}

\subsection{Proof of Theorem~\ref{thm:law_large_numbers_weighted_sums}}

For a fixed $n$, let $\widetilde{X_{\gri}}
=A_{n,\gri}\pr{X_{\gri}}$. By linearity of $A_{n,\gri}$, 
$\pr{\widetilde{X_{\gri}}}_{\gri\in\Z^d}$ is 
an orthomartingale difference random field. Moreover, 
the following inequality takes place:
\begin{equation}\label{eq:convex_bound_weighted_sum}
 \sum_{\gri\in\Z^d}\norm{\widetilde{X_{\gri}}}_{\B}^p
 \conv C_{n,p}\norm{X_{\gr{1}}}_{\B}^p.
\end{equation}
Indeed, let $\varphi\colon\R\to\R$ be a convex non-decreasing
function. Then using the fact that $\varphi$ is non-decreasing 
and the elementary bound 
$\norm{Ax}_{\B}\leq\norm{A}_{\Bca\pr{\B}}\norm{x}_{\B}$, we derive that  
\begin{equation}
 \E{\varphi\pr{\sum_{\gri\in\Z^d}\norm{\widetilde{X_{\gri}}}_{\B}^p}}\leq 
\E{\varphi\pr{\sum_{\gri\in\Z^d}\frac{
\norm{A_{n,\gri}}_{\Bca\pr{\B}}^p}{C_{n,p}}C_{n,p}
\norm{ X_{\gri}}_{\B}^p} }
\end{equation}
hence convexity of $\varphi$ and the fact that random
variables $\norm{X_{\gri}}_{\B}^p, \gri\in\Z^d$, have the 
same distribution gives \eqref{eq:convex_bound_weighted_sum}. 

We are now in position to apply Corollary~\ref{cor:dev_ortho_stoch_dom} to 
$q=s+1$ and $V=C_n^{1/p}\norm{X_{\gr{1}}}_{\B}$, which gives 
\begin{multline}
 \PP\pr{\norm{\sum_{\gri\in\Z^d}A_{n,\gri}\pr{X_{\gri}} }_{\B}>\varepsilon 
C_n^{1/p} R_n}
\leq f_{p,s,d}\pr{C_{p,\B}}
\int_0^1u^{s}\PP\pr{\norm{X_{\gr{1}}}_{\B}>\varepsilon R_nu }du\\
+f_{p,s,d}\pr{C_{p,\B}}
\int_1^\infty 
u^{p-1}\pr{1+\log u}^{d}\PP\pr{\norm{X_{\gr{1}}}_{\B}>\varepsilon R_nu }du.
\end{multline}
From the elementary (in)equalities 
\begin{align}
 \sum_{n\geq 1}\pr{R_n^s-R_{n-1}^s}\PP\pr{Y>R_n}
 &=\sum_{n\geq 1}\sum_{k\geq n}
 \pr{R_n^s-R_{n-1}^s}\PP\pr{R_k<Y\leq R_{k+1}}\\
 &= \sum_{k\geq 1}\sum_{n: k\geq n}
 \pr{R_n^s-R_{n-1}^s}\PP\pr{R_k<Y\leq R_{k+1}}\\
 &\leq \sum_{k\geq 1}R_k^s\PP\pr{R_k<Y\leq R_{k+1}}\\
 &\leq \sum_{k\geq 1}
 \E{Y^s \ind{R_k<Y\leq R_{k+1}}}\leq \E{Y^s },
\end{align}
 we derive that 
 \begin{multline}
  \sum_{n\geq 1}\pr{R_n^s-R_{n-1}^s} 
\PP\pr{\norm{\sum_{\gri\in\Z^d}A_{n,\gri}\pr{X_{\gri}} }_{\B}>\varepsilon 
C_n^{1/p} R_n}\leq 
f_{p,s,d}\pr{C_{p,\B}}
\int_0^1u^{s}\E{\norm{X_{\gr{1}}}_{\B}^s \pr{\varepsilon u}^{-s} } du\\
+f_{p,s,d}\pr{C_{p,\B}}
\int_1^\infty 
u^{p-1}\pr{1+\log u}^{d}\E{\norm{X_{\gr{1}}}_{\B}^s \pr{\varepsilon u}^{-s} }du
 \end{multline}
and since $s>p$, the last integral is convergent. This ends the proof of 
Theorem~\ref{thm:law_large_numbers_weighted_sums}. 

\begin{appendices}
\section{Appendix}
\subsection{Properties of orthomartingales}

 \begin{Proposition}\label{prop:moments_ordre_r_orthomartingale_Banach}
 Let $\pr{\B,\norm{\cdot}_{\B}}$ be an $r$-smooth Banach space.
 There exists a constant $C\pr{\B}$ such that for each
 $d\geq 1$ and each orthomartingale difference
 random field $\pr{X_{\gri}}_{\gri\in\Z^d}$,
 \begin{equation}\label{eq:moments_ordre_r_orthomartingale_Banach}
  \E{\norm{\sum_{\gr{1}\imd \gri\imd\grn} X_{\gri} }_{\B}^r }
  \leq C\pr{\B}^d
\sum_{\gr{1}\imd \gri\imd\grn}
\E{\norm{X_{\gri}}_{\B}^r}.
 \end{equation}
 \end{Proposition}

\begin{proof}
 We proceed by induction: for $d=1$, this is exactly
\eqref{eq:moments_ordre_r_martingale_Banach}. Assume now that
\eqref{eq:moments_ordre_r_orthomartingale_Banach} is true for any
$d$ dimensional orthomartingale difference random field and let
$\pr{X_{\gri;i_{d+1}}}_{\gri\in\Z^d,i_{d+1}\in\Z}$ be a
$\pr{d+1}$-dimensional
$\B$-valued orthomartingale difference random field. Since
$\pr{\sum_{i_{d+1}=1}^{n_{d+1}} X_{\gri;i_{d+1}} }_{\gri\in\Z^d}$ is an
orthomartingale difference random field, the induction assumption implies
that
\begin{equation}
 \E{\norm{\sum_{\gr{1}\imd \gri\imd\grn}\sum_{i_{d+1}=1}^{n_{d+1}}
X_{\gri,i_{d+1}} }_{\B}^r }
  \leq C\pr{\B}^d\sum_{\gr{1}\imd \gri\imd\grn}
\E{  \norm{\sum_{i_{d+1}=1}^{n_{d+1}}   X_{\gri,i_{d+1} }}_{\B}^r}
\end{equation}
and we conclude by an other application of
\eqref{eq:moments_ordre_r_martingale_Banach}.
\end{proof}

\subsection{Some operators}

\begin{Definition}\label{def:operateur_pq}
 For $-\infty\leq p<q\leq \infty$ and $d\in\N$, denote by $\Hca_{p,q,d}$
 the set of functions
  $g\colon\R_+\to\R_{+}$ such that
$\int_0^\infty\min\ens{u^{p-1},u^{q-1}}\pr{1+\ind{u>1}\log
u}^dg\pr{u}du<\infty$.
 We define the operator $T_{p,q,d}$ defined on
$\Hca_{p,q}$ by
\begin{equation} \label{eq:def_operateur_pq}
T_{p,q,d}\pr{g}\pr{x}:=
 \int_{0}^\infty \min\ens{u^{p-1},u^{q-1}}\pr{1+\ind{u>1}\log
u}^dg\pr{xu}du.
\end{equation}
 \end{Definition}
 Since $p<q$, we have
 \begin{equation}
T_{p,q,d}\pr{g}\pr{x}=
 \int_0^1 u^{q-1}g\pr{xu}du+ \int_1^\infty u^{p-1}\pr{1+\log u}^dg\pr{xu}du.
\end{equation}
 Note that here, $p$ is allowed to be equal to $-\infty$, in which case,
 \begin{equation}\label{eq:dfn_T_moins_inf_q}
  T_{-\infty,q,d}\pr{g}\pr{x}=
 \int_{0}^1 u^{q-1} g\pr{xu}du
 \end{equation}
and similarly, the case $q=\infty$ is also possible to consider, giving
 \begin{equation}
  T_{p,\infty,d}\pr{g}\pr{x}=
 \int_1^\infty u^{p-1}\pr{1+\log u}^dg\pr{xu}du .
 \end{equation}

 Most of the deviation inequalities in this paper take the form
 $f\pr{x} \leq T_{p,q,d}\pr{g}\pr{x}$, where $f$ and $g$ are tail functions of
some random variables. Repeated applications of such inequalities lead us to
consider composition of operators $T_{p,q,d}$ with different parameters.
 \begin{Lemma}\label{lem:operateur_pq}
 Let $p_1,q_1,p_2$ and $q_2$ be such that $-\infty\leq p_1,
p_2<\infty$,
 $q_1\neq q_2$, $-\infty<q_1,q_2\leq  \infty$
and $\max\ens{p_1,p_2}<\min\ens{q_1,q_2}$. Let $d_1,d_2\in\N$,
$p=\max\ens{p_1,p_2}$, $q=\min\ens{q_1,q_2}$ and
$d=d_1+d_2+\ind{p_1=p_2}$. Assume that $p<q$. Then for
each function $g$ in
$\Hca_{p,q,d}$ and each positive $x$,
\begin{equation}\label{eq:iteration_Tpq}
 T_{p_1,q_1,d_1}\circ T_{p_2,q_2,d_2}\pr{g}\pr{x}
 \leq C_{p_1,q_1,p_2,q_2}
T_{p, q ,d}\pr{g}\pr{x},
\end{equation}
where
\begin{equation}
 C_{p_1,q_1,p_2,q_2,d_1,d_2} =I_{q-p,d_1+d_2}2^{d_1+d_2}
 \pr{\frac{\ind{p_1\neq p_2} }{\abs{p_1-p_2 }}
+\ind{p_1=p_2} +\frac{1}{\abs{q_1-q_2}}},
\end{equation}
\begin{equation}\label{eq:def_de_Iqd}
  I_{q,d}=\int_1^{\infty}v^{-1-q}\pr{1+\log v}^d dv
\end{equation}

and when one of the numbers $p_1$ or $p_2$ is $-\infty$ or $q_1$ or $q_2$ is
$\infty$, the corresponding fraction in $C_{p_1,q_1,p_2,q_2,d_1,d_2} $ is
understood as $0$, as well as $\ind{p_1\neq p_2} /\abs{p_1-p_2 }$ if $p_1=p_2$.
\end{Lemma}
 \begin{proof}
  We will treat the case where all the numbers $p_1,q_1,p_2$ and $q_2$ are
finite. The general case can be deduced by monotone convergence.
By definition of $T_{p,q,d}$,
  \begin{multline}
   T_{p_1,q_1,d_1}\circ T_{p_2,q_2,d_2} g\pr{x}
\\=\int_{0}^\infty\int_{0}^\infty\min\ens{u^{p_1-1},u^{q_1-1}}
\pr{1+\ind{u>1} \log u}^{d_1}
\min\ens{v^{p_2-1},v^{ q_2-1 } }\pr{1+\ind{v>1} \log v}^{d_2}
   g\pr{uvx}dvdu.
  \end{multline}
Then doing the substitution $t=uv$ for a fixed $u$ gives
  \begin{multline}
   T_{p_1,q_1,d_1}\circ T_{p_2,q_2,d_2} g\pr{x}   \\
=\int_{0}^\infty\int_{0}^\infty\min\ens{u^{p_1-1},u^{q_1-1}}\pr{1+\ind{u>1}
\log
u}^{d_1}
\min\ens{\pr{\frac{t}u}^{p_2-1},\pr{\frac{t}u}^{ q_2-1 } }\pr{1+\ind{t>u} \log
\frac tu}^{d_2}\frac 1u
   g\pr{tx}dtdu.
  \end{multline}
  and we are reduced to bound
  \begin{equation}
f\pr{t}= \int_{0}^\infty\min\ens{u^{p_1-1},u^{q_1-1}}\pr{1+\ind{u>1} \log
u}^{d_1}
\min\ens{\pr{\frac{t}u}^{p_2-1},\pr{\frac{t}u}^{ q_2-1 } }\pr{1+\ind{t>u} \log
\frac tu}^{d_2}\frac 1u.
  \end{equation}
To do so, we split the integral according to the cases
$t\leq 1$ or not and $u\leq t$ or not in order to get
\begin{multline}
f\pr{t}=f_1\pr{t}+f_2\pr{t}+f_3\pr{t}+f_4\pr{t}\\
:=\int_{0}^\infty \ind{u>1}\ind{t> u}u^{p_1-1}\pr{\frac{t}u}^{p_2-1}
\pr{1+\log u}^{d_1}\pr{1+\log\pr{\frac tu}}^{d_2}
 \frac
1u    du
\\ +\int_{0}^\infty \ind{u>1}\ind{t\leq u}u^{p_1-1}
\pr{1+\log u}^{d_1}
\pr{\frac{t}u}^{q_2-1} \frac
1u    du
 +\int_{0}^\infty\ind{u\leq 1 } \ind{t>u}
u^{q_1-1}\pr{\frac{t}u}^{p_2-1}\pr{1+\log\pr{\frac tu}}^{d_2}
du \\
+\int_{0}^\infty\ind{u\leq 1} \ind{t\leq u}  u^{q_1-1}\pr{\frac{t}u}^{q_2-1}
du.
\end{multline}
Then $f_i$, $1\leq i\leq 4$ can be bounded as follows:
\begin{align*}
f_1\pr{t} &=\ind{t>1}t^{p_2-1}\int_{1}^t
u^{p_1-p_2- 1}\pr{1+\log u}^{d_1}\pr{1+\log\pr{\frac tu}}^{d_2}du  \\
&\leq
\ind{t>1}t^{p_2-1}\pr{1+\log
t}^{d_1+d_2}\pr{ \ind{p_1\neq p_2}\frac{t^{p_1-1}-t^{p_2-1}}{ p_1-p_2 }
+\ind{p_1=p_2}t^{p_2-1}\log t}\\
&\leq \ind{t>1}t^{p-1}\pr{1+\log
t}^{d_1+d_2}\pr{ \frac{\ind{p_1\neq p_2} }{\abs{p_1-p_2 }}
+\ind{p_1=p_2} \log t},
\end{align*}
\begin{align*}
 f_2\pr{t}&=t^{q_2-1}\int_{\max\ens{1,t}}^\infty u^{p_1-q_2-1}\pr{1+\log
u}^{d_1}du\\
&=t^{q_2-1}\max\ens{1,t} ^{p_1-q_2}\int_{1}^\infty
v^{p_1-q_2-1}\pr{1+\log\pr{v\max\ens{1,t} }}^{d_1}du\\
 &\leq 2^{d_1-1}I_{q_2-p_1,d_1}t^{q_2-1}\max\ens{1,t} ^{p_1-q_2}
 \pr{1+\log\pr{ \max\ens{1,t}}  }^{d_1}\\
 &=\ind{t\leq 1} 2^{d_1-1}I_{q-p,d_1}t^{q-1}+
 \ind{t>1} 2^{d_1-1}I_{q-p,d_1}t^{p-1}
 \pr{1+\log\pr{ t}  }^{d_1},
\end{align*}
where $I_{q,d}$ is defined as in \eqref{eq:def_de_Iqd},
\begin{align*}
f_3\pr{t} &= t^{p_2-1}\int_{0}^{\min\ens{1,t}}u^{ q_1-p_2-1}\pr{1+\log\pr{\frac
tu}}^{d_2}    du
\\
&\leq 2^{d_2}I_{q_1-p_2,d_2} t^{p_2-1}
  \min\ens{1,t}^{ q_1-p_2}\pr{1+\log\pr{\frac{t}{\min\ens{1,t}}}}^{d_2}\\
&= \ind{t\leq 1}2^{d_2}I_{q-p,d_2} t^{q-1}
+\ind{t>1}2^{d_2}I_{q-p,d_2} t^{p-1}
   \pr{1+\log\pr{t}}^{d_2}
\end{align*}
where we did the substitution $v=\min\ens{1,t}/u$ and
 \begin{equation}
  f_4\pr{t}=\ind{t\leq 1}t^{q_2-1}\int_t^1 u^{q_1-q_2-1}du
  =\ind{t\leq 1}t^{q -1}\frac{1}{\abs{q_1-q_2}}
 \end{equation}
hence
\begin{multline}
 f\pr{t}\leq I_{q-p,d_1+d_2}2^{d_1+d_2}\ind{t\leq 1}t^{q-1}
     \frac{ 1 }{\abs{q_1-q_2}}
   \\
  + I_{q-p,d_1+d_2}2^{d_1+d_2} \ind{t>1} t^{p-1}\pr{1+\log
t}^{d_1+d_2}\pr{ \frac{\ind{p_1\neq p_2} }{\abs{p_1-p_2 }}
+\ind{p_1=p_2} \log t  },
\end{multline}
and
bounding the constant terms by $C\pr{p_1,q_1,p_2,q_2,d_1,d_2}$ gives
\eqref{eq:iteration_Tpq}, which ends the proof of Lemma~\ref{lem:operateur_pq}.

 \end{proof}

\subsection{Tail inequalities}

In order to state the tail inequalities, we will make a use of the operators
$T_{p,q,d}$ given in Definition~\ref{def:operateur_pq}. In order to ease the
notation, we will use the notation $\tail{Y}$ for the tail function of the
non-negative random variable $Y$, that is,
\begin{equation}\label{eq:def_tail}
 \tail{Y}\colon t\mapsto\PP\pr{Y>t}.
\end{equation}
A substitution shows that for a positive $s$ and $t$,
\begin{equation}\label{eq:T_pq_tail_powers}
 T_{p,q,d}\pr{\tail{Y^s}}\pr{t^s}=sT_{sp,sq,d}\pr{\tail{Y}}\pr{t}.
\end{equation}

\subsubsection{Doob's inequality}

\begin{Lemma}\label{lem:Doob}
Let $\pr{Y_n}_{n\geq 1}$ be a non-negative submartingale with respect to 
a filtration $\pr{\Fca_n}_{n\geq 0}$. Then for each positive $x$, 
\begin{equation}
\tail{\max_{1\leq n\leq N}Y_n}\pr{t}\leq 4T_{1,\infty,
0}\pr{\tail{Y_N}}\pr{t/4}.
\end{equation}
\end{Lemma}
\begin{proof}
By classical Doob's inequality, 
\begin{equation}
t\tail{\max_{1\leq n\leq N}Y_n}\pr{t}
\leq \E{Y_N\ind{\max_{1\leq n\leq N}Y_n>t}}
\end{equation}
and splitting this expectation according to $Y_N\leq t/2$ or not, we derive that 
\begin{equation}
t\tail{\max_{1\leq n\leq N}Y_n}\pr{t}
\leq 2\E{Y_N\ind{Y_N>t/2}}=2\int_0^\infty \PP\pr{Y_N>\max\ens{t/2,s}}\D{s}
\end{equation}
hence 
\begin{equation}
t\tail{\max_{1\leq n\leq N}Y_n}\pr{t}
\leq t\PP\pr{Y_N>t/2}+t\int_{1}^\infty\PP\pr{Y_N>t/2s}\D{s}
\end{equation}
and using $t/4\PP\pr{Y_N>t/2}\leq \int_{t/4}^{t/2}\PP\pr{Y_N>s}\D{s}$ 
completes the proof.
\end{proof}

\subsubsection{Tail inequality for conditional moments}

\begin{Lemma}\label{lem:tail_ineq_cond}
 Let $\pr{\Fca_i}_{i=0}^N$ be an increasing sequence of
 sub-$\sigma$-fields on a probability space $\pr{\Omega,\Fca,\PP}$. Let
$\pr{Y_i}_{i\geq
 1}$ be a sequence of non-negative random variables such that $Y_i$ is
$\Fca_i$-measurable for each $i\geq 1$ and define the random variables
$S_n^{\operatorname{cond}}=\sum_{i=1}^n \E{Y_i\mid\Fca_{i-1}}$ and
$S_n=\sum_{i=1}^n Y_i$.
For each $q\geq 1$ and  $t>0$, the following inequality takes place:
\begin{equation}\label{eq:inegalite_varcib}
 \PP\pr{S_N^{\operatorname{cond}}>t}\leq \kappa\pr{q}
 \int_0^\infty \min\ens{1,v^{q-1}}
 \PP \ens{S_N>tv}dv.
\end{equation}
where $\kappa\pr{q}$ depends only on $q$. In other words,
\begin{equation}\label{eq:inegalite_varcib_Tpq}
 \tail{S_N^{\operatorname{cond}}}\pr{t}\leq
 \kappa\pr{q}T_{1,q,0}\pr{\tail{S_N}}\pr{t},
\end{equation}
where $T_{1,q,0}$ is defined as in \eqref{eq:def_operateur_pq}.
\end{Lemma}
\begin{proof}

 We will first give a functional inequation for the tail of $Y_\infty$ in terms
of quantitites of the form $\E{Z\mathbf{1}_{\ens{Z>u}}}$. We will first follow
 the arguments of \cite{zbMATH03538599}, p. 175, but since we are dealing with
 finite sums of random variables, we do not need to introduce stopping times.
 Instead, we will first show that for each positive $t$,

 \begin{equation}\label{eq:control_sum_cond_tail_2}
 \E{\pr{S_N^{\operatorname{cond}}-t}\1{S_N^{\operatorname{cond}}>t}}\leq
\E{S_N\1{S_N^{\operatorname{cond}}>t}}.
\end{equation}
To do so, define for each $n\geq 1$ the events $A_n=\ens{Y_n>t}$ and
$B_{n}=A_n\setminus A_{n-1}$, $n\geq 2$ and $B_1=A_1$. Notice that $B_n$
is $\Fca_{n-1}$-measurable for each $n$ and $\ens{S_N^{\operatorname{cond}}>t}$
is the disjoint union of the sets $B_n$, $1\leq n\leq N$. Letting $Y_0=0$, we
also have that $Y_{n-1}\1{B_n}\leq t\ipr{B_n}$ hence
\begin{equation*}
 \E{\pr{S_N^{\operatorname{cond}}-t}\1{S_N^{\operatorname{cond}}>t}}
=\sum_{n=1}^N\E{\pr{S_N^{\operatorname{cond}}-t}\ipr{B_n}}
\leq \sum_{n=1}^N\E{\pr{S_N^{\operatorname{cond}}-S_{n-1}}\ipr{B_n}}.
\end{equation*}
Since
\begin{equation}
 \E{\pr{S_N^{\operatorname{cond}}-S_{n-1}}\ipr{B_n}\mid \Fca_{n-1} }
 = \E{\sum_{i=n}^N \E{Y_i\mid\Fca_{i-1}}\ipr{B_n}\mid \Fca_{n-1} }
\end{equation}
and $B_n$ is $\Fca_{n-1}$-measurable, one gets
\begin{align}
 \E{\pr{S_N^{\operatorname{cond}}-S_{n-1}}\ipr{B_n}\mid \Fca_{n-1} }
 &= \ipr{B_n}\E{\sum_{i=n}^N \E{Y_i\mid\Fca_{i-1}}\mid \Fca_{n-1} }\\
 &=\ipr{B_n}\E{\sum_{i=n}^N  Y_i\mid\Fca_{n-1}}  \\
 &\leq \ipr{B_n}\E{S_N\mid\Fca_{n-1}}.
\end{align}
Consequently,
\begin{equation}
 \E{\pr{S_N^{\operatorname{cond}}-t}\1{S_N^{\operatorname{cond}}>t}}
\leq \sum_{n=1}^N\E{\ipr{B_n}\E{S_N\mid\Fca_{n-1}}}.
\end{equation}
Using again $\Fca_{n-1}$-measurability of $B_n$ and the fact that
$\ens{S_N^{\operatorname{cond}}>t}$ is the disjoint union of the sets $B_n$,
$1\leq n\leq N$, we get \eqref{eq:control_sum_cond_tail_2}. Now, from the
inequalities
\begin{equation}
 \E{S_N^{\operatorname{cond}}\1{S_N^{\operatorname{cond}}>2t}}\leq
2\E{\pr{S_N^{\operatorname{cond}}-t}\1{S_N^{\operatorname{cond}}>2t}}
 \leq 2\E{\pr{S_N^{\operatorname{cond}}-t}\1{S_N^{\operatorname{cond}}>t}}
\end{equation}
and \eqref{eq:control_sum_cond_tail_2}, we derive that
\begin{equation}
  \E{S_N^{\operatorname{cond}}\1{S_N^{\operatorname{cond}}>2t}}\leq
2\E{S_N\1{S_N^{\operatorname{cond}}>t}}
\end{equation}
and spliting the expectation of $S_N\1{S_N^{\operatorname{cond}}>t}$ over the
set
$\ens{S_N\leq \delta t}$ and its complement, we arrive at the estimate
\begin{equation}\label{eq:delta_inequality}
 \PP\pr{S_N^{\operatorname{cond}}>2t}\leq \delta
\PP\pr{S_N^{\operatorname{cond}}>t}+\E{\frac{S_N}t\1{S_N>t\delta}},
\end{equation}
valid for each $t>0$ and each positive $\delta$. In order to get a bound
like in the right hand side of \eqref{eq:inegalite_varcib}, let us define for a
fixed $s>0$ and $k$ the numbers $a_k:=\PP\pr{S_N^{\operatorname{cond}}>2^st}$,
$b_k=\E{\frac{S_N}{2^ks}\1{S_N>2^ks\delta}}$
and $c_k=\delta^{-k}a_k$. Then \eqref{eq:delta_inequality} applied with
$t=2^ks$ translates as
\begin{equation}
 c_{k+1}=\delta^{-k-1}a_{k+1}\leq \delta^{-k-1}\pr{\delta a_k+b_k}
\end{equation}
hence $c_{k+1}-c_k\leq \delta^{-k-1}b_k$. It follows that
\begin{equation}
 c_n=c_0+\sum_{k=0}^{n-1}\pr{c_{k+1}-c_k}\leq
a_0+\sum_{k=0}^{n-1}\delta^{-k-1}b_k.
\end{equation}
Multiplying by $\delta^n$, this bound translates as
\begin{equation}
 \PP\pr{S_N^{\operatorname{cond}}>2^ns}
 \leq  \delta^n\PP\pr{S_N^{\operatorname{cond}}>t}
 +\sum_{k=0}^{n-1}\delta^{n-k-1}\E{\frac{S_N}{2^ks}\1{S_N>2^st\delta}}
\end{equation}
and applying this to $t=2^ns$ gives
\begin{equation}
 \PP\pr{S_N^{\operatorname{cond}}>t}
 \leq  \delta^n\PP\pr{S_N^{\operatorname{cond}}>t2^{-n}}
+\sum_{k=0}^{n-1}\delta^{n-k-1}\E{\frac{S_N}{2^{k-n}t}\1{S_N>2^kt2^{-n}\delta}}.
\end{equation}
The change of index $j=n-k$ leads to
\begin{equation}
 \PP\pr{S_N^{\operatorname{cond}}>t}
 \leq  \delta^n\PP\pr{S_N^{\operatorname{cond}}>u2^{-n}}
 +\sum_{j=0}^{n}\delta^{j-1}2^j\E{\frac{S_N}{ t}\1{S_N>2^{-j}t \delta}}.
\end{equation}
Letting $n$ going to infinity furnishes the estimate
\begin{equation}
 \PP\pr{S_N^{\operatorname{cond}}>t}
 \leq  \sum_{j=0}^{\infty}\delta^{j}2^j\E{\frac{S_N}{\delta t}\1{S_N>2^{-j}t
\delta}}.
\end{equation}
Choosing $\delta=2^{-q-1}$  gives
\begin{equation}
 \PP\pr{S_N^{\operatorname{cond}}>t}
 \leq  \sum_{j=0}^{\infty} 2^{-qj}\E{\frac{S_N}{ \delta t}\1{S_N>2^{-j}t
\delta}},
\end{equation}
and using the elementary inequality
\begin{equation}
 \sum_{j\geq 0}2^{-jq}\1{Y>2^{-j}}\leq C_q \1{Y>1}+C_q
 Y^{q-1}\1{Y\leq 1}
\end{equation}
with $Y=S_N/\pr{\delta t}$ gives, after having expressed the
expectation as an integral of the tail, the wanted inequality
\eqref{eq:inegalite_varcib}. Then \eqref{eq:inegalite_varcib_Tpq} follows from
a the substitution $v=e^u$. This ends the proof of
Lemma~\ref{lem:tail_ineq_cond}.
\end{proof}

We formulate an inequality of the spirit of Theorem~1.3 in \cite{MR4046858},
except that the term involving the sum of conditional moments of order $p$ is
replaced by the unconditional sum, which will turn out to be more convenient in
our context.

\begin{Proposition}\label{prop:deviation_martingales_power}
Let $1<r\leq 2$ and let $\pr{\B,\norm{\cdot}_{\B}}$ be a separable $r$-smooth
Banach space. For each $p\in (1,r]$, and $q>p$, there exists
a function $f_{p,q}\colon \R_+\to \R_+$ such that if $\pr{D_i}_{i\geq 1}$ is
a  $\B$-valued martingale difference sequence with respect to the filtration
$\pr{\Fca_{i}}_{i\in\Z}$  then for each
$1<p\leq r$,
$q>p$ and  $x>0$, the following inequality holds:
\begin{multline}\label{eq:deviation_martingale}
 \PP\pr{\max_{1\leq n\leq N}\norm{\sum_{i=1 }^n D_{i }  }_{\B} >t
 }\\
  \leq f_{p,q}\pr{C_{p,\B}}
 \int_0^\infty \min\ens{u^{q-1},u^{p-1}} \mathbb P\pr{
 \pr{\sum_{i=1}^N\norm{D_{i}}_{\B}^p}^{1/p}>tu  }
\mathrm{d}u.
\end{multline}
\end{Proposition}

\begin{proof}[Proof of Proposition~\ref{prop:deviation_martingales_power}]
 Let $\pr{\B,\norm{\cdot}_{\B}}$ be a separable
$r$-smooth Banach space, where $1<r\leq 2$.

According to  Theorem~1.3 in \cite{MR4046858}, for each $1<p\leq r$, each
$q>0$, $t>0$ and each $\B$-valued martingale differences sequence
$\pr{X_i,\Fca_i}_{i\geq 1}$, the following inequality holds:
 \begin{multline}\label{eq:deviation_inequality_non_stationary_martingale_diff}
\tail{  \max_{1\leqslant i\leqslant n}\norm{\sum_{j=1}^i X_j}_{\B} }
\pr{t}
    \leq \frac{2^{q}}{2^q-1}q2^{-p}T_{-\infty,q,0}
    \pr{\tail{\max_{1\leqslant i\leqslant
n}
    \norm{X_i}_{\B}}}  \pr{2^{-1-q/p}C_{p,\B}^{-1/p}t}
\\+\frac{2^q}{2^q-1}q2^{-p}T_{-\infty,q,0}\pr{
   \tail{ \pr{\sum_{i=1}^n
    \E{\norm{X_i}_{\B}^{p}\mid
\Fca_{i-1}}}^{1/p}}}\pr{2^{-1-q/p}C_{p,\B}^{-1/p}t} ,
   \end{multline}
 where $T_{-\infty,q,0}$ is defined as in \eqref{eq:dfn_T_moins_inf_q}
 and $\tail{\cdot}$ by \eqref{eq:def_tail}.
Bounding $\max_{1\leqslant i\leqslant
n}\norm{X_i}_{\B}$ by $\pr{\sum_{i=1}^n \norm{X_i}_{\B}^p}^{1/p}$, we infer that

   \begin{multline}
\label{eq:deviation_inequality_non_stationary_martingale_diff_ter}
   \tail{\max_{1\leqslant i\leqslant n}\norm{\sum_{j=1}^i X_j}_{\B}}\pr{t}
\leq \frac{2^{q-p}}{2^q-1}q
T_{0,q,0}\pr{\tail{\pr{\sum_{i=1}^n\norm{X_i}_{\B}^p}^{1/p}}}\pr{2^{-1-q/p}C_{p,
\B }^ {
-1/p }t}     \\+\frac{2^{q-p}}{2^q-1}q
T_{0,q,0}\pr{\tail{ \sum_{i=1}^n\E{\norm{X_i}_{\B}^p\mid\Fca_{i-1}  }
}}\pr{2^{-p-q}C_{p,\B}^{-1}t^p } =:A_1+A_2
   \end{multline}

 Therefore, we are reduced
to bound the last term of
\eqref{eq:deviation_inequality_non_stationary_martingale_diff_ter} by
an other one involving the tails of $\sum_{i=1}^n \norm{X_i}^p $. This
is done with the help of Lemma~\ref{lem:tail_ineq_cond} used in the
following setting: $Y_i=\norm{X_i}_{\B}$, $\widetilde{q}=q-p$ and
$\widetilde{t}=2^{-p-q}C_{p,\B}^{-1}t^pu^p$ . This allows the bound $A_2$ by
\begin{multline}
K\pr{ p, q,C_{p,\B}}
T_{0,q-p,0}\circ T_{1,q,0}\pr{\tail{ \sum_{i=1}^n \norm{X_i}_{\B}^p
}}\pr{2^{-p-q}C_{p,\B}^{-1}t^p }\\
=K\pr{ p, q,C_{p,\B}}p
T_{0,q-p,0}\circ T_{1,q,0}\pr{\tail{\pr{ \sum_{i=1}^n \norm{X_i}_{\B}^p}^{1/p}
}}\pr{2^{-1-q/p}C_{p,\B}^{-1/p}t },
\end{multline}
where the equality comes from \eqref{eq:T_pq_tail_powers}.
%
%

Then an application of Lemma~\ref{lem:operateur_pq} with $p_1=0$, $q_1=q-p$,
$p_2=1$,$q_2=q$ and $d_1=d_2=0$ gives the wanted result, as
$T_{1,q,0}\pr{g}\pr{x}\geq
T_{0,q,0}\pr{g}\pr{x}$. The proof of
Proposition~\ref{prop:deviation_martingales_power} is complete.
\end{proof}

\subsubsection{Convex ordering}
The following ordering was studied in \cite{MR606989}.
\begin{Definition}
 Let $X$ and $Y$ be two real-valued random variables. We say that $X\conv Y$ if
for each nondecreasing convex $\varphi\colon\R\to\R$ such that the expectations
$\E{\varphi\pr{X}}$ and $\E{\varphi\pr{Y}}$ exist, $\E{\varphi\pr{X}}
\leq \E{\varphi\pr{Y}}$.
\end{Definition}

The point of this ordering is that if $X\conv Y$, one can formulate a tail
inequality for $X$ in terms of tails of $Y$. More precisely, Lemma~2.1 in
\cite{MR4294337} gives that if $X$ and $Y$ are nonnegative random variables
such that $X\conv Y$, then for each $t$,
\begin{equation}\label{eq:conv_ordering_tails_integrales}
 \PP\pr{X>t}\leq \int_1^\infty\PP\pr{Y>tv/4}dv.
\end{equation}
In terms of operators $T_{p,q,d}$ defined as in
Definition~\ref{def:operateur_pq}, this reads
\begin{equation}\label{eq:conv_ordering_tails}
 X\conv Y\Rightarrow \tail{X}\pr{t}\leq T_{1,\infty,0}\pr{\tail{Y}}\pr{t/4}.
\end{equation}

\subsection{Technical tools}

Let $Y$ be a non-negative random variable and denote
$L\pr{x}=1+\abs{\log x}$. The following inequalities will be used in the sequel:
\begin{equation}\label{eq:sum_indicators_2skkd_leq}
 \sum_{k=0}^{+\infty}2^{k\pr{1-r/p}}\pr{k+1}^{d-1}\mathbf 1\ens{Y\leq
2^{k/p}}\leq c_{p,d}
\pr{1+L\pr{Y}}^{d-1}Y^{p-r}, p<r \mbox{ and }
\end{equation}
\begin{equation}\label{eq:sum_indicators_2skkd}
 \sum_{k=0}^{+\infty} 2^{sk} \pr{k+1}^{d-1}\mathbf{1}
 \ens{Y >2^{k/p}   }
\leq c_{p,s,d}
Y^{ps} L\pr{Y}^{d-1}\mathbf{1}\ens{Y>1} .
\end{equation}
 
Inequality \eqref{eq:sum_indicators_2skkd_leq} follows 
from the observation that the sum can be restricted to 
the indexes $k$ bigger than $p\log\pr{Y}/\log 2$.
 
Inequality \eqref{eq:sum_indicators_2skkd} follows from the 
following steps:
\begin{align*}
 \sum_{k=0}^{+\infty} 2^{sk} \pr{k+1}^{d-1}\mathbf{1}
 \pr{Y >2^{k/p}   }&= \sum_{k=0}^{+\infty} 2^{sk} \pr{k+1}^{d-1}\sum_{j\geq k} 
\mathbf{1}
 \ens{2^{j/p}<Y \leq 2^{\pr{j+1}/p}   }\\
 &=\sum_{j=0}^{+\infty} \sum_{k=0}^j 2^{sk} \pr{k+1}^{d-1}
\mathbf{1}
 \ens{2^{j/p}<Y \leq 2^{\pr{j+1}/p}   }\\
 &\leq C_{s,d}\sum_{j=0}^{+\infty}  2^{sj} j^{d-1}
\mathbf{1}
 \ens{2^{j/p}<Y \leq 2^{\pr{j+1}/p}   }\\
 &\leq c_{p,s,d}\sum_{j=0}^{+\infty} Y^{ps} L\pr{Y}^{d-1}
\mathbf{1}
 \ens{2^{j/p}<Y \leq 2^{\pr{j+1}/p}   }
\end{align*}

\end{appendices}

\def\cprime{$'$}
\providecommand{\bysame}{\leavevmode\hbox to3em{\hrulefill}\thinspace}
\providecommand{\MR}{\relax\ifhmode\unskip\space\fi MR }
\providecommand{\MRhref}[2]{%
  \href{http://www.ams.org/mathscinet-getitem?mr=#1}{#2}
}
\providecommand{\href}[2]{#2}


\begin{thebibliography}{{Nev}75}

\bibitem[Ass75]{MR0407963}
P.~Assouad, \emph{Espaces {$p$}-lisses et {$q$}-convexes, in\'{e}galit\'{e}s de
  {B}urkholder}, S\'{e}minaire {M}aurey-{S}chwartz 1974--1975: {E}spaces
  {$L^{p}$}, applications radonifiantes et g\'{e}om\'{e}trie des espaces de
  {B}anach, {E}xp. {N}o. {XV}, 1975, p.~8. \MR{0407963}

\bibitem[Bur73]{MR365692}
D.~L. Burkholder, \emph{Distribution function inequalities for martingales},
  Ann. Probability \textbf{1} (1973), 19--42. \MR{365692}

\bibitem[CW75]{MR420845}
R.~Cairoli and John~B. Walsh, \emph{Stochastic integrals in the plane}, Acta
  Math. \textbf{134} (1975), 111--183. \MR{420845}

\bibitem[DLP13]{MR3077911}
J. Ding, J.~R. Lee, and Y. Peres, \emph{Markov type and threshold
  embeddings}, Geom. Funct. Anal. \textbf{23} (2013), no.~4, 1207--1229.
  \MR{3077911}

\bibitem[DM07]{MR2743029}
J.~Dedecker and F.~Merlev{\`e}de, \emph{Convergence rates in the law of large
  numbers for {B}anach-valued dependent variables}, Teor. Veroyatn. Primen.
  \textbf{52} (2007), no.~3, 562--587. \MR{2743029 (2011k:60096)}

\bibitem[EM07]{MR2269603}
M. El~Machkouri, \emph{Nonparametric regression estimation for random
  fields in a fixed-design}, Stat. Inference Stoch. Process. \textbf{10}
  (2007), no.~1, 29--47. \MR{2269603}

\bibitem[Faz05]{MR2264866}
I. Fazekas, \emph{Burkholder's inequality for multiindex martingales},
  Ann. Math. Inform. \textbf{32} (2005), 45--51. \MR{2264866 (2008a:60105)}

\bibitem[FGL15]{MR3311214}
X. Fan, I. Grama, and Q. Liu, \emph{Exponential inequalities for
  martingales with applications}, Electron. J. Probab. \textbf{20} (2015), no.
  1, 22. \MR{3311214}

\bibitem[FGL17]{MR3579898}
\bysame, \emph{Deviation inequalities for martingales with applications}, J.
  Math. Anal. Appl. \textbf{448} (2017), no.~1, 538--566. \MR{3579898}

\bibitem[Gir19]{MR4046858}
D. Giraudo, \emph{Deviation inequalities for {B}anach space valued
  martingales differences sequences and random fields}, ESAIM Probab. Stat.
  \textbf{23} (2019), 922--946. \MR{4046858}

\bibitem[Gir21a]{MR4294337}
\bysame, \emph{An exponential inequality for {$U$}-statistics of {I}.{I}.{D}.
  data}, Teor. Veroyatn. Primen. \textbf{66} (2021), no.~3, 508--533.
  \MR{4294337}

\bibitem[Gir21b]{MR4186670}
\bysame, \emph{Maximal function associated to the bounded law of the iterated
  logarithms via orthomartingale approximation}, J. Math. Anal. Appl.
  \textbf{496} (2021), no.~1, Paper No. 124792, 25. \MR{4186670}

\bibitem[Gut78]{MR494431}
A. Gut, \emph{Marcinkiewicz laws and convergence rates in the law of large
  numbers for random variables with multidimensional indices}, Ann. Probability
  \textbf{6} (1978), no.~3, 469--482. \MR{494431}

\bibitem[Hoe63]{MR144363}
W. Hoeffding, \emph{Probability inequalities for sums of bounded random
  variables}, J. Amer. Statist. Assoc. \textbf{58} (1963), 13--30. \MR{144363}

\bibitem[JS88]{MR995572}
W.~B. Johnson and G.~Schechtman, \emph{{Martingale inequalities in
  rearrangement invariant function spaces}}, Israel J. Math. \textbf{64}
  (1988), no.~3, 267--275 (1989). \MR{995572 (90g:60048)}

\bibitem[KL11]{MR2794415}
A. Kuczmaszewska and Z.~A. Lagodowski, \emph{Convergence rates in the
  {SLLN} for some classes of dependent random fields}, J. Math. Anal. Appl.
  \textbf{380} (2011), no.~2, 571--584. \MR{2794415}

\bibitem[KVW16]{MR3483738}
J. Klicnarov\'a, D. Voln\'y, and Y. Wang, \emph{Limit theorems for
  weighted {B}ernoulli random fields under {H}annan's condition}, Stochastic
  Process. Appl. \textbf{126} (2016), no.~6, 1819--1838. \MR{3483738}

\bibitem[Lag16]{MR3451971}
Z.~A. Lagodowski, \emph{An approach to complete convergence theorems for
  dependent random fields via application of {F}uk-{N}agaev inequality}, J.
  Math. Anal. Appl. \textbf{437} (2016), no.~1, 380--395. \MR{3451971}

\bibitem[LXW13]{MR3114713}
W. Liu, H. Xiao, and W.~B. Wu, \emph{Probability and moment
  inequalities under dependence}, Statist. Sinica \textbf{23} (2013), no.~3,
  1257--1272. \MR{3114713}

\bibitem[M\'78]{MR520006}
C.~M\'{e}traux, \emph{Quelques in\'{e}galit\'{e}s pour martingales \`a
  param\`etre bidimensionnel}, S\'{e}minaire de {P}robabilit\'{e}s, {XII}
  ({U}niv. {S}trasbourg, {S}trasbourg, 1976/1977), Lecture Notes in Math., vol.
  649, Springer, Berlin, 1978, pp.~170--179. \MR{520006}

\bibitem[Nag82]{MR669054}
S.~V. Nagaev, \emph{Probability inequalities for sums of independent random
  variables with values in a {B}anach space}, Limit theorems of probability
  theory and related questions, Trudy Inst. Mat., vol.~1, ``Nauka'' Sibirsk.
  Otdel., Novosibirsk, 1982, pp.~159--167, 208. \MR{669054}

\bibitem[Nag03]{MR2021875}
\bysame, \emph{On probability and moment inequalities for supermartingales and
  martingales}, Proceedings of the {E}ighth {V}ilnius {C}onference on
  {P}robability {T}heory and {M}athematical {S}tatistics, {P}art {II} (2002),
  vol.~79, 2003, pp.~35--46. \MR{2021875 (2005f:60098)}

\bibitem[{Nev}75]{zbMATH03538599}
J.~{Neveu}, \emph{{Discrete-parameter martingales. Translated by T. P. Speed}},
  vol.~10, Elsevier (North-Holland), Amsterdam, 1975 (English).

\bibitem[PV20]{MR4125956}
M. Peligrad and D. Voln\'{y}, \emph{Quenched invariance principles for
  orthomartingale-like sequences}, J. Theoret. Probab. \textbf{33} (2020),
  no.~3, 1238--1265. \MR{4125956}

\bibitem[Rue81]{MR606989}
L. R\"{u}schendorf, \emph{Ordering of distributions and rearrangement of
  functions}, Ann. Probab. \textbf{9} (1981), no.~2, 276--283. \MR{606989}

\bibitem[Rio00]{MR2117923}
E.~Rio, \emph{{Th{\'e}orie asymptotique des processus al{\'e}atoires faiblement
  d{\'e}pendants}}, {Math{\'e}matiques \& Applications (Berlin) [Mathematics \&
  Applications]}, vol.~31, Springer-Verlag, Berlin, 2000. \MR{2117923
  (2005k:60001)}

\bibitem[Rio09]{MR2472010}
E.~Rio, \emph{Moment inequalities for sums of dependent random variables
  under projective conditions}, J. Theoret. Probab. \textbf{22} (2009), no.~1,
  146--163. \MR{2472010 (2010d:60043)}

\bibitem[Smy73]{MR346881}
R.~T. Smythe, \emph{Strong laws of large numbers for {$r$}-dimensional arrays
  of random variables}, Ann. Probability \textbf{1} (1973), no.~1, 164--170.
  \MR{346881}

\bibitem[Vol15]{MR3427925}
D. Voln\'{y}, \emph{A central limit theorem for fields of martingale
  differences}, C. R. Math. Acad. Sci. Paris \textbf{353} (2015), no.~12,
  1159--1163. \MR{3427925}

\bibitem[Vol19]{MR3913270}
\bysame, \emph{On limit theorems for fields of martingale differences},
  Stochastic Process. Appl. \textbf{129} (2019), no.~3, 841--859. \MR{3913270}

\bibitem[WW13]{MR3222815}
Y. Wang and M. Woodroofe, \emph{A new condition for the invariance
  principle for stationary random fields}, Statist. Sinica \textbf{23} (2013),
  no.~4, 1673--1696. \MR{3222815}

\end{thebibliography}
\end{document}